\documentclass[reqno]{amsart}
\usepackage{amsmath,amssymb,amsthm}
\usepackage{color,rotating}
\usepackage{psfrag,graphicx}
\usepackage{hyperref}

\usepackage[all]{xy}

\newcommand{\Z}{\ensuremath{\mathbb{Z}}} 
\newcommand{\R}{\ensuremath{\mathbb{R}}} 

\newcommand{\Core}{\ensuremath{\mathcal{C}}} 
\newcommand{\clop}[1]{\ensuremath{[\![#1]\!]}} 
\newcommand{\bd}{\ensuremath{\partial}} 

\newcommand{\iv}{\ensuremath{o}} 
\newcommand{\tv}{\ensuremath{t}} 

\newcommand{\uc}{\widetilde} 
\newcommand{\co}{\colon\thinspace} 

\newcommand{\CV}{\ensuremath{\mathcal{CV}_k}} 
\newcommand{\cv}{\ensuremath{cv_k}} 
           
\newcommand{\lambdap}{\mu}

\newcommand{\inv}{^{-1}}

\theoremstyle{plain}
\newtheorem{intro-theorem}{Theorem}
\newtheorem{theorem}{Theorem}[section]
\newtheorem{lemma}[theorem]{Lemma}
\newtheorem{proposition}[theorem]{Proposition}
\newtheorem*{prop-estimate}{Proposition \ref{prop:estimate}}
\newtheorem{corollary}[theorem]{Corollary}
\newtheorem{intro-corollary}[intro-theorem]{Corollary}

\theoremstyle{definition}
\newtheorem{definition}[theorem]{Definition}
\newtheorem{notation}[theorem]{Notation}
\newtheorem{example}[theorem]{Example}
\newtheorem{remark}[theorem]{Remark}
\newtheorem{intro-remark}[intro-theorem]{Remark}
\newtheorem{convention}[theorem]{Convention}

\DeclareMathOperator{\Aut}{Aut}
\DeclareMathOperator{\Out}{Out}

\DeclareMathOperator{\GL}{GL}
\DeclareMathOperator{\vol}{vol}

\begin{document}


\title{Growth of intersection numbers for free group automorphisms}
\author[J~.Behrstock]{Jason Behrstock}
\address{Dept.\ of Mathematics \\
   Lehman College\\
   Bronx, NY 10468}
\email{jason.behrstock@lehman.cuny.edu}

\author[M.~Bestvina]{Mladen Bestvina}
\address{Dept.\ of Mathematics \\
   University of Utah\\
   155 South 1400 East, JWB 233 \\
   Salt Lake City, UT 84102}
\email{bestvina@math.utah.edu}

\author[M.~Clay]{Matt Clay}
\address{Dept.\ of Mathematics\\
  University of Oklahoma\\
  Norman, OK 73019} 
\email{mclay@math.ou.edu}

\date{\today}
\thanks{Partially supported by NSF grants DMS-0812513 and DMS-0502441.}

\begin{abstract}
For a fully irreducible automorphism $\phi$ of the free group $F_k$ we
compute the asymptotics of the intersection number $n\mapsto
i(T,T'\phi^n)$ for trees $T,T'$ in Outer space. We also obtain
qualitative information about the geometry of the Guirardel core for
the trees $T$ and $T'\phi^n$ for $n$ large.
\end{abstract}

\maketitle


\section*{Introduction}\label{sc:intro}

Parallels between $\GL_n(\Z)$, the mapping class group,
$\mathcal{MCG}(\Sigma)$, and the outer automorphism group of a free
group, $\Out(F_k)$, drive much of the current research of these groups
and is particularly fruitful in the case of $\Out(F_k)$. The article
\cite{col:BV06} lists many similarities between these groups and uses
known results in one category to generate questions in another.  A
significant example of this pedagogy is the question of the existence
of a complex useful for studying the large scale geometry of
$\Out(F_k)$ analogous to the spherical Tits building for $\GL_n(\Z)$
or Harvey's curve complex \cite{harvey:boundary} for the mapping class group.

The curve complex is a simplicial complex whose vertices correspond to
homotopy classes of essential simple closed curves and whose simplices
encode when curves can be realized disjointly on the surface. The
curve complex has played a large role in the study of the mapping
class group; one of the first major results was the computation of its
homotopy type and its consequences for homological stability,
dimension and duality properties of the mapping class group
\cite{Harer,Harer2}. Another fundamental result is that the
automorphism group of the curve complex is the (full) mapping class
group (except for some small complexity cases)
\cite{ivanov:complexes2,korkmaz:complex,luo:complex}. The curve
complex has also played an important role in understanding the large
scale geometry of the mapping class group \cite{BKMM}, a key property
there is that the curve complex is Gromov hyperbolic
\cite{masurminsky:complex1}.

The situation with $\Out(F_k)$ seems to be much more complicated and
an emphasis on a particular feature of the curve complex leads to a
different analog.  A discussion of some of these analogs and their
basic properties is provided in \cite{KapovichLustig}. Without much
doubt, a construction of such a complex and a proof of its
hyperbolicity is the central question in the study of $\Out(F_k)$
today. In this introduction we will feature three candidate complexes.

Recall that an (outer) automorphism of $F_k$ is {\it fully
  irreducible} (sometimes called {\it irreducible with irreducible
  powers (iwip)}) if no conjugacy class of a proper free factor is
periodic. These automorphisms are analogous to pseudo-Anosov
homeomorphisms in mapping class groups.  Further recall that Culler
and Vogtmann's Outer space, $\CV$, is the space of minimal free
simplicial (metric) $F_k$--trees normalized such that the volume of
the quotient graph is 1 \cite{ar:CV86}.  We consider the
unprojectivized version as well, $\cv$.

\medskip
\noindent {\bf The complex of free factors of a free group.} An
$n$-simplex in this complex is a chain $FF_0<FF_1<\cdots<FF_n$ of
nontrivial proper free factors in $F_k$ modulo simultaneous
conjugacy. Hatcher and Vogtmann showed that this complex has the
homotopy type of a bouquet of spheres, a result which is analogous to
that of the spherical Tits building \cite{ar:HV98} and of the curve
complex \cite{Harer:MCGhomologystability}.  By analogy with the curve
complex situation where pseudo-Anosov homeomorphisms have unbounded
orbits and other homeomorphisms have bounded orbits, Kapovich and
Lustig have shown that fully irreducible automorphisms act with
unbounded orbits and other automorphisms with bounded orbits
\cite{KapovichLustig}. Kapovich and Lustig proved their result via a
notion of intersection number using geodesic currents analogous to a
construction of Bonahon's in the surface case \cite{Bonahon:ends}.

\medskip
\noindent {\bf The complex of (connected) subgraphs.} A vertex is a
nontrivial proper free factor of $F_k$ modulo conjugacy. A collection
$FF_i$ of such free factors spans a simplex if they are compatible:
there is a filtered graph $G_0\subset G_1\subset\cdots\subset G_m=G$
representing $F_k$ so that each $FF_i$ is represented by a connected
component of some $G_i$. Just like the collection of very short curves
in a hyperbolic surface (if nonempty) determines a simplex in the
curve complex, leading to the homotopy equivalence between the thin
part of Teichm\"uller space and the curve complex, so does the core of
the union of very short edges of a marked metric graph (if nonempty)
determine a simplex in the complex of subgraphs, leading to the
homotopy equivalence between this complex and the thin part of Outer
space. There is a natural embedding of the free factor complex into
the subgraph complex (the vertex sets are equal) and this embedding is
a quasi-isometry.

\medskip \noindent {\bf The splitting complex.} This complex is a
refinement of the complex of free factors, where one also provides a
complementary free factor. More precisely, this is the simplicial
complex whose vertex set is the set of free product decompositions
$F_k = A*B$ (modulo simultaneous conjugation and switching of the
factors), where $n+1$ free product decompositions span an $n$--simplex
if they are pairwise \emph{compatible}. Two free product
decompositions $A*B$ and $A'*B'$ are compatible if there is a two-edge
graph of groups decomposition $F_k = X*Y*Z$ such that collapsing one
edge yields the decomposition $A*B$ and collapsing the other edge
yields the decomposition $A'*B'$. The motivation for studying this
complex comes from the observation that an essential simple closed
curve on a surface determines a splitting of the fundamental group
over $\Z$ (and not just one of the factors). Moreover, as we explain
below, there is a hope that a proof of hyperbolicity of the curve
complex generalizes to the splitting complex.

Scott and Swarup have shown that compatibility of $A*B$ and $A'*B'$
can be interpreted as the vanishing of an intersection number
$i(A*B,A'*B')$ \cite{ar:SS00}.  This number $i(-,-)$ is defined for
any two splittings, either as an amalgamated free product or as an
HNN-extension, of a finitely generated group.  When the group is the
fundamental group of a surface and the splittings arise from simple
closed curves on the surface, this intersection number agrees with the
geometric intersection between the two curves.

Guirardel, incorporating work of \cite{un:CLS} and \cite{ar:FP06},
generalized Scott's intersection number to the setting of $G$--trees
\cite{ar:G05}.  More importantly, given two $G$--trees $T,T'$,
Guirardel constructed a \emph{core} $\Core(T \times T') \subset T
\times T'$.  This core specifies a geometry for the pair of splittings
in the sense that it is CAT(0) and it is equipped with two patterns
respectively representing the splittings such that the geometric
intersection number between the patterns in the quotient $\Core(T
\times T')/G$ is the intersection number $i(T,T')$.  When the group is
the fundamental group of a closed surface and the splittings arise
from filling simple closed curves, the quotient $\Core(T \times T')/G$
is the same surface endowed with the singular euclidean structure used
by Bowditch in his proof of hyperbolicity of the curve complex
\cite{ar:B06}. It seems promising that following a careful
understanding of the geometry of the core for free product
decompositions of $F_k$, Bowditch's proof of the hyperbolicity of the
curve complex may also show hyperbolicity of the splitting complex. In
fact, the goal of this paper is much more modest. Instead of
attempting to understand the geometry of the core for a pair of free
splittings we restrict ourselves to two points in Outer space, and the
two points differ by a high power of a fully irreducible automorphism.

One of the main differences between the mapping class group and
$\Out(F_k)$ is the inherent asymmetry present in $\Out(F_k)$. This
difference arises as a mapping class is represented by a homeomorphism
of a surface, a symmetric object, whereas in general an outer
automorphism of $F_k$ is not represented by a homeomorphism of a graph
but merely a homotopy equivalence, which has no symmetry imposed. The
asymmetry is most easily seen in \emph{expansion factors}.  A
pseudo-Anosov homeomorphism of a surface has two measured foliations,
a stable and unstable foliation, and an expansion factor $\lambda$
such that the homeomorphism multiplies the measure on the unstable
foliation by the factor $\lambda$ and the measure on the stable
foliation by the factor $\lambda\inv$ \cite{ar:T88}.  For the inverse
of the homeomorphism, the role of the foliations change place and the
expansion factor $\lambda$ remains the same.  For a fully irreducible
automorphism there is a homotopy equivalence of metric graphs
$\sigma\co \Gamma \to \Gamma$ representing the automorphism that maps
vertices to vertices and linearly expands each edge of $\Gamma$ by the
same number $\lambda$, known as the \emph{expansion factor}, and all
positive powers of $\sigma$ are locally injective on the interior of
every edge (such maps are called {\it train-track maps}
\cite{ar:BH92}). However, unlike the surface case, there is no reason
why the expansion factor for an automorphism to equal the expansion
factor of its inverse.  Indeed the following automorphism and it
inverse provide an example where the two expansion factors are not
equal:
\begin{align*}
  & a \mapsto b & & a \mapsto cA \\
  \phi\co & b \mapsto c &  \phi\inv\co & b \mapsto a \\
  & c \mapsto ab & & c \mapsto b
\end{align*}
The expansion factor for $\phi$ is approximately 1.32 and the
expansion factor for $\phi\inv$ is approximately 1.46.  We remark that
Handel and Mosher have shown that the ratio between the logarithms of
the expansion factors is bounded by a constant depending on $k$
\cite{ar:HM07}.

Intersection numbers for $G$--trees, as intersection numbers for
curves on a surface, are symmetric (first proved by Scott in
\cite{ar:S98}, obvious for Guirardel's construction).  They are also
invariant under automorphisms: $i(T,T') = i(T\phi,T'\phi)$ for any
$G$--trees $T$ and $T'$ and an automorphism $\phi$ of $G$ (compare to
$i(\alpha,\beta) = i(\psi(\alpha),\psi(\beta))$ for any curves
$\alpha,\beta$ on a surface and a mapping class $\psi$.).  This
imposes a symmetry on free group automorphisms.  In particular, for
any $F_k$--tree $T$ and $\phi \in \Out(F_k)$ one has $i(T,T\phi^n) =
i(T\phi^{-n},T) = i(T,T\phi^{-n})$ for all $n$.  This naturally leads
one to inquire about the asymptotic behavior of the function $n
\mapsto i(T,T'\phi^n)$.

In the surface setting, for a pseudo-Anosov homeomorphism $\psi$ and
any curves $\alpha$ and $\beta$ the function $n \mapsto
i(\alpha,\psi^n(\beta))$ behaves like $n \mapsto \lambda^n$ where
$\lambda$ is the expansion factor of $\psi$ \cite{ar:T88}.  This leads
one to first guess that for a fully irreducible automorphism $\phi$
and any two $F_k$--trees $T,T'$ the function $i(T,T'\phi^n)$ also
behaves like $\lambda^n$ where $\lambda$ is the expansion factor of
$\phi$.  But as stated above, this cannot possibly be true in general
since the asymptotics of $i(T,T'\phi^n)$ and $i(T,T'\phi^{-n})$ are
the same but the expansion factors of $\phi$ and $\phi^{-1}$ need not
be the same.

To state the correct result we remind the reader about the
\emph{stable tree} associated to a fully irreducible automorphism. For
a fully irreducible automorphism $\phi \in \Out(F_k)$ and any tree $T
\in \CV$ the sequence $T\phi^n$ has a well-defined limit in the
compactification of Outer space called the stable tree
\cite{ar:BFH97}.  The stable tree is a nonsimplicial (projectivized)
$\R$--tree.  It is \emph{geometric} if it is dual to a measured
foliation on a 2--complex \cite{un:BF,ar:LP97}.  Geometricity of the
stable tree is characterized by a train-track representative for
$\phi$ \cite{un:BF}; this characterization appears as Theorem
\ref{th:geo-nongeo}.  For two functions $f,g\co \R \to \R$ we write $f
\sim g$ to mean that there are constants $K,C$ such that
$\frac{1}{K}f(x) - C \leq g(x) \leq Kf(x) + C$.

\begin{intro-theorem}\label{th:growth}
  Suppose $\phi \in \Out(F_{k})$ is fully irreducible with expansion
  factor $\lambda$ and $T,T' \in \cv$.  Let $T^+$ be the stable tree
  for $\phi$ and let $\lambdap$ be the expansion factor of
  $\phi^{-1}$.
 
\begin{itemize}
	
\item[1.]  If $T^+$ is geometric, then $i(T,T'\phi^n)
  \sim \lambda^n$; else
  
\item[2.]  if $T^+$ is nongeometric, then $i(T,T'\phi^n) \sim
  \lambda^n + \lambda^{n-1}\lambdap + \cdots + \lambda\lambdap^{n-1} +
  \lambdap^n$.
\end{itemize}
\end{intro-theorem}

\begin{intro-remark}\label{rm:growth}
  Case 2 in Theorem \ref{th:growth} can be simplified in two ways
  depending on whether or not $\lambda = \lambdap$.  If $\lambda \neq
  \lambdap$ then $i(T,T'\phi^n) \sim \max\{ \lambda,\lambdap\}^n$.  If
  $\lambda = \lambdap$ then $i(T,T'\phi^n) \sim n\lambda^n$.
\end{intro-remark}

As a corollary we get a statement about the expansion factors for
certain fully irreducible automorphisms.  This corollary was first
proved by Handel and Mosher.

\begin{intro-corollary}[\cite{un:HM}]\label{co:HM}
  Suppose $\phi \in \Out(F_{k})$ is fully irreducible with expansion
  factor $\lambda$.  Let $T^+,T^-$ be the stable trees for
  $\phi,\phi\inv$ respectively and let $\lambdap$ be the expansion
  factor of $\phi^{-1}$.  If $T^+$ is geometric and $T^-$ is
  nongeometric then $\lambda > \lambdap$.
\end{intro-corollary}

\begin{proof}
  If $\phi$ is as in the hypotheses of the corollary then applying
  Theorem \ref{th:growth} we get: $\lambda^n \sim \lambda^n +
  \lambda^{n-1}\lambdap + \cdots + \lambda\lambdap^{n-1} +
  \lambdap^n$. Therefore $\lambda > \lambdap$.
\end{proof}

By \cite[Corollary 9.3]{ar:G05} or \cite[Corollary 3]{un:HM}, if both
$T^+$ and $T^-$ are geometric then $\phi$ is realized by a
homeomorphism of a punctured surface.  Therefore, the automorphism
$\phi$ appearing in Corollary \ref{co:HM} is a \emph{parageometric}
automorphism, i.e.~an automorphism with a geometric stable tree that
is not realized by a homeomorphism of a punctured surface.

\begin{intro-remark}
In view of this discussion, Theorem \ref{th:growth} can be more simply
stated as saying that:
\begin{itemize}
\item[1.] $i(T,T'\phi^n)\sim \max\{\lambda,\lambdap\}^n$ if
  $\lambda\neq\lambdap$ or if $\phi$ is realized on a surface,
\item[2.] $i(T,T'\phi^n)\sim n\lambda^n$ if $\lambda=\lambdap$ and
  $\phi$ is not realized on a surface.
\end{itemize}
\end{intro-remark}

We briefly outline the rest of the paper.  In Section \ref{sc:core} we
collect all of the necessary properties of Guirardel's core that we
need.  Our first step toward Theorem \ref{th:growth} is showing that
the asymptotics of $i(T,T'\phi^n)$ do not depend on the trees $T,T'
\in \cv$ in Section \ref{sc:build}.  Following this, in Section
\ref{sc:slice} we derive an algorithm for computing Guirardel's core
in terms of the \emph{map on ends} associated to an automorphism and
present an example.  In Section \ref{sc:consolidation}, we further
refine our algorithm to show that $i(T,T'\phi^n)$ is comparable to the
subtree spanned by $(f^n)\inv(p)$ where $f\co T \to T$ is a lift of a
special representative for $\phi$ called a \emph{train-track
  representative} (Proposition \ref{prop:estimate}).  The appearance
of the inverse of the train-track map explains the appearance of both
expansion factors in the second case of Theorem \ref{th:growth}.
Further in Section \ref{sc:consolidation} we give a quick proof
showing that $i(T,T'\phi^n)$ is asymptotically bounded below by
$\max\{\lambda,\lambdap\}^n$.

The rest of the proof of Theorem \ref{th:growth} follows from the
analysis in Section \ref{sc:growth} of the subtree $T^n_e$ in
Proposition \ref{prop:estimate}.  There are two cases depending on
whether the stable tree $T^+$ is geometric (Propositions
\ref{prop:ngvolslice} and \ref{prop:geo-vp}).  Putting together the
three aforementioned propositions, we get Theorem \ref{th:growth}.
Finally we present examples of the core for both an automorphism with
a geometric stable tree and an automorphism with a nongeometric stable
tree highlighting the difference.

We conclude our introduction by recalling some standard definitions.

\medskip
\noindent {\bf $\mathbb{R}$--trees:} An \textit{$\R$--tree} $T$ (or
simply \textit{tree}) is a metric space such that any two points $x,y
\in T$ are the endpoints of a unique arc and this arc is isometric to
an interval of $\R$ of length $d_T(x,y)$.  In particular every
$\R$--tree has a Lebesgue measure.  A point $p$ is called a
\textit{branch point} if the number of connected components of $T
\setminus \{p\}$ is greater than 2.  An \R--tree is a simplicial tree
if the set of branch points is discrete.  In this case, we denote the
originating and terminating vertices of an oriented edge $e \subset T$
by $\iv(e)$ and $\tv(e)$ respectively and the edge with reverse
orientation by $\bar{e}$.  A $G$--tree is a tree with an isometric
action of $G$.  We identify two $G$--trees if there is a
$G$--equivariant isometry between them.  A $G$--tree is
\textit{trivial} if there exists a global fixed point.  We will always
assume that $G$ is finitely generated and that $G$--trees are minimal
(no proper invariant subtree).

We will briefly recall the definition of the boundary of a tree $T$.
A \textit{ray} is an isometry $R\co [0,\infty) \to T$.  Two rays are
equivalent if their images lie in a bounded neighborhood of one
another; an equivalence class of rays is called an \textit{end}.  The
equivalence class of the ray $R$ is denoted $R_\infty$.  A
\textit{geodesic} is an isometry $\rho\co(-\infty,\infty) \to T$ to
which are associated two ends denoted $\rho_\infty$ and
$\rho_{-\infty}$.  We will often identify a ray or geodesic with its
image in $T$.  The \textit{boundary} of $T$, denoted $\bd T$, is the
set of ends.  If a basepoint $p \in T$ is fixed, the set of ends can
be identified with the set of rays that originate at $p$.

For a $G$--tree $T$ it is well known that every element $g \in G$
either fixes a point in $T$ (elliptic) or else it has an axis $A_{g}$
(a set isometric to \R) on which it acts by translation (hyperbolic).
The set of rays $R \subset A_{g}$ such that $gR \subset R$
defines a unique point $\omega_{T}(g) \in \bd T$.

\medskip \noindent {\bf Morphism in $\cv$:} For trees $T,T' \in \cv$,
by \emph{morphism} $f\co T \to T'$ we mean an equivariant cellular map
that linearly expands every edge of $T$ over a tight edge path in
$T'$.  This definition of morphism differs from the usual definition
for $\R$--trees.  The notion of \emph{bounded cancellation} arises
frequently when dealing with automorphisms of free groups.  There are
many statements of bounded cancellation, the one of importance to us
in this section is that any morphism $f\co T \to T'$ is a
\emph{quasi-isometry}.  In particular, there exist positive constants
$K,C$ such that $\frac{1}{K}d_T(x,y) - C \leq d_{T'}(f(x),f(y)) \leq
Kd_T(x,y) + C$.  The original statement of bounded cancellation, along
with a proof can be found in \cite{ar:C87}.  A morphism $f\co T \to
T'$ descends to a (linear) Lipschitz homotopy equivalence $\sigma\co
\Gamma \to \Gamma'$ where $\Gamma = T/F_k$ and $\Gamma' = T'/F_k$.
Also, a morphism $f\co T \to T'$ induces an equivariant homeomorphism
$f_\infty\co \bd T \to \bd T'$ called the \emph{map on ends}.

\medskip \noindent {\bf Notation:} The edge path obtained by
tightening an edge path $\alpha$ relative to its endpoints is denoted
by $[\alpha]$ and the concatenation of two edge paths $\alpha$ and
$\beta$ is denoted $\alpha \cdot \beta$.

\medskip \noindent {\bf Train-track representatives:} We recall the
basics of train-tracks, see \cite{ar:BH92} for proofs.  For a metric
graph $\Gamma$, a cellular homotopy equivalence $\sigma\co\Gamma \to
\Gamma$ that is linear on edges is a \emph{train-track map} if there
is a collection $\mathcal{L}$ of unordered pairs of distinct germs of
adjacent edges such that
\begin{itemize}
\item[1.] $\mathcal{L}$ is closed under the action of $\sigma$ and
\item[2.] for an edge $e \subset \Gamma$, any pair of germs crossed by
  $\sigma(e)$ is in $\mathcal{L}$.
\end{itemize}
The unordered pairs in $\mathcal{L}$ are called \emph{legal turns}, an
unordered pair of distinct germs of adjacent edges not in
$\mathcal{L}$ is called an \emph{illegal turn}.  An edge path is
\emph{legal} if it only crosses legal turns.  There is a metric on
$\Gamma$ such that $\sigma$ linearly expands each edge of $\Gamma$ by
the same factor $\lambda$, called the \emph{expansion factor}.  This
factor is the Perron--Frobenious eigenvalue of the transition matrix
for $\sigma$, a positive eigenvector for this eigenvalue specifies the
metric on $\Gamma$.  Bounded cancellation implies that there is a
bound on the amount of cancellation when tightening $\sigma(\alpha
\cdot \beta)$ where $\alpha$ and $\beta$ are legal paths.  We denote
the optimal constant by $BCC(\sigma)$.  As such, when $\alpha,\beta$
and $\gamma$ are legal paths, if $length(\beta) >
\frac{2BCC(\sigma)}{\lambda - 1}$ then the length of $[\sigma^n(\alpha
\cdot \beta \cdot \gamma)]$ goes to infinity as $n \to \infty$.  The
number $\frac{2BCC(\sigma)}{\lambda -1}$ is called the \emph{critical
  constant} for the map $\sigma$.

\medskip \noindent {\bf Nielsen paths:} A {\it Nielsen path} $\gamma$
is a tight path with $[\sigma(\gamma)]=\gamma$. A tight path $\gamma$
is a {\it periodic Nielsen path} if $[\sigma^n(\gamma)]=\gamma$ for
some $n>0$ and $[\sigma^i(\gamma)]$, $i=0,1,\cdots,n-1$ is an {\it
  orbit} of (periodic) Nielsen paths. In this paper we always consider
periodic Nielsen paths and for convenience usually omit the adjective
``periodic''. A Nielsen path is {\it indivisible} if it is not a
concatenation of nontrivial Nielsen subpaths, and similarly for orbits
of Nielsen paths. An indivisible Nielsen path (or iNp) always has a
unique illegal turn and the two legal pieces have equal lengths.

\bigskip \noindent {\bf Acknowledgements.}  The authors would like to
thank the referee for suggestions which improved this work.


\section{The Guirardel core}\label{sc:core}

In this section we describe Guirardel's construction of a core
$\Core(T \times T') \subset T \times T'$ for two $G$--trees $T$ and
$T'$.  In addition, we introduce some notation and state some of the
basic properties of cores needed for the following.  Roughly speaking,
the core is the essential part of $T \times T'$ in terms of the
diagonal action of $G$ on the product $T \times T'$.  The following
definitions and remarks appear in \cite{ar:G05}.

\begin{definition}\label{def:direction}
  Let $T$ be a tree and $p$ a point in $T$.  A \textit{direction at
    $p$} is a connected component of $T - \{p\}$.  If $T$ is
  simplicial and $e$ is an oriented edge of $T$, we will use the
  notation $\delta_{e}$ to denote the direction at $\iv(e)$ that
  contains $e$.  Given two trees $T$ and $T'$, a \textit{quadrant} is
  a product $\delta \times \delta'$ where $\delta \subset T$ and
  $\delta' \subset T'$ are directions.
\end{definition}

For a direction $\delta \subset T$, let $\delta_{\infty} \subset \bd
T$ denote the set of ends determined by rays contained in $\delta$.

\begin{definition}\label{def:heavy}
  Let $T$ and $T'$ be $G$--trees and fix a basepoint $(p,p') \in T
  \times T'$.  A quadrant $\delta \times \delta' \subset T \times T'$
  is \textit{heavy} if there exists a sequence of elements $g_i \in G$
  such that:
  \begin{itemize}

  \item[1.]  $g_i (p,p') \in \delta \times \delta'$; and
 
  \item[2.] $d_T(p,g_i p) \to \infty$ and $d_{T'}(p',g_i p') \to
    \infty$ as $i \to \infty$.

  \end{itemize}
  If the quadrant is not heavy, it is \textit{light}.
\end{definition}

\begin{remark}\label{rm:basepoint} 
  The choice of basepoint is irrelevant.  In particular, if the
  intersection of a quadrant $\delta \times \delta'$ with the orbit of
  any point is a bounded set (or more generally has a bounded
  projection), then $\delta \times \delta'$ is light.
\end{remark}

\begin{remark}\label{rm:inclusion}
  Suppose we have an inclusion of quadrants $\delta \times \delta'
  \subseteq \eta \times \eta' \subset T \times T'$.  If $\delta \times
  \delta'$ is heavy, then $g_{i}(p,p') \in \delta \times \delta'
  \subseteq \eta \times \eta'$ for some sequence of elements $g_{i}
  \in G$.  As the second condition of Definition \ref{def:heavy} does
  not depend on the quadrant, we see that the quadrant $\eta \times
  \eta'$ is heavy.
\end{remark}

We can now define Guirardel's core.  It is the part of $T \times T'$
that is not in any light quadrant.

\begin{definition}[\textit{The Guirardel Core}]\label{def:core}
  Let $T$ and $T'$ be minimal $G$--trees.  Let $\mathcal{L}(T \times
  T')$ be the collection of light quadrants in the product $T \times
  T'$.  The \textit{core} $\Core(T \times T')$ is defined as:
  \begin{equation*}\label{eq:core}
    \Core(T \times T') = T \times T' - \bigcup_{\delta \times
      \delta' \in \mathcal{L}(T \times T')} \delta \times \delta'.
  \end{equation*}
\end{definition}

Since the collection $\mathcal{L}(T \times T')$ of light quadrants is
invariant with respect to the diagonal action of $G$ on $T \times T'$,
the group $G$ acts on the core $\Core(T \times T')$.  Guirardel
defines the \textit{intersection number} between two $G$--trees $T$
and $T'$, denoted $i(T,T')$, as the volume of $\Core(T \times T')/G$.
The measure on $\Core(T \times T')$ is induced from the product
Lebesgue measure on $T \times T'$.  If $T$ and $T'$ are simplicial
then the intersection number $i(T,T')$ is the sum of the areas of the
2--cells in $\Core(T \times T')/G$.

Guirardel shows that this intersection number agrees with the usual
notion of intersection number for simple closed curves on a surface
when $T$ and $T'$ are the Bass--Serre trees for the splitting of the
surface group associated to the curves.  Also, Guirardel shows that
this intersection number agrees with Scott's definition of
intersection number for splittings \cite{ar:S98} and relates the core
$\Core(T \times T')$ to various other constructions, see \cite{ar:G05}
for references and further motivation for this definition.

We now state several properties of the core proved by Guirardel that
we need for the following.  

\begin{proposition}[Proposition 2.6 \cite{ar:G05}]\label{prop:convex}
  Let $T$ and $T'$ be $G$--trees.  The core $\Core(T \times T')$ has
  convex fibers, i.e. the sets $\Core(T \times T') \cap \{x\} \times
  T'$ and $\Core(T \times T') \cap T \times \{x'\}$ are connected
  (possibly empty) for any $x \in T$ and $x' \in T'$.
\end{proposition}

Recall that a $G$--tree is \textit{irreducible} if there exist two
hyperbolic elements whose axes intersect in a compact set
\cite{ar:CM87}.  In particular any tree in $\cv$ is irreducible.

\begin{proposition}[Proposition 3.1 \cite{ar:G05}]\label{prop:empty}
  Let $T$ and $T'$ be $G$--trees.  If either $T$ or $T'$ is
  irreducible, then $\Core(T \times T') \neq \emptyset$.  In
  particular if $T,T' \in \cv$ then $\Core(T \times T') \neq
  \emptyset$.
\end{proposition}

In \cite{ar:G05}, the above proposition is more general, but for our
purposes, the above version is sufficient.  Whenever the core is
non-empty, there is a stronger condition for heavy quadrants.

\begin{proposition}[Corollary 3.8 \cite{ar:G05}]\label{prop:heavy2}
  Let $T$ and $T'$ be $G$--trees and suppose $\Core(T \times T') \neq
  \emptyset$.  Then a quadrant $\delta \times \delta'$ is heavy if and
  only if there is an element $g \in G$ which is hyperbolic for both
  $T$ and $T'$ such that $\omega_{T}(g) \in \bd \delta$ and
  $\omega_{T'}(g) \in \bd \delta'$.
\end{proposition}

It may happen that the core is not connected.  In this case there is a
procedure of adding ``diagonal'' edges to the core, resulting in the
\emph{augmented core} $\hat{\Core}$.  Adding edges to $\Core$ does not
effect the volume of $\Core/G$ and hence does not effect the
intersection number.


\section{\texorpdfstring{The core $\mathcal{C}(T \times T')$ for trees
    in $cv_k$}{The core C(T x T') for trees in cvk}}\label{sc:build}

In this section we give a method for computing the intersection number
between two trees in $\cv$ and show that the asymptotics of $n \mapsto
i(T,T'\phi^n)$ do not depend on the trees $T$ and $T'$.

\begin{lemma}\label{lm:heavy}
  Let $f\co T \to T'$ be a morphism between $T,T' \in \cv$.  For two
  directions $\delta \subset T$ and $\delta' \subset T'$ the quadrant
  $\delta \times \delta'$ is heavy if and only if
  $f_{\infty}(\delta_{\infty}) \cap \delta'_{\infty} \neq \emptyset$.
\end{lemma}

\begin{proof}
  Suppose that $\delta \times \delta'$ is heavy.  By Proposition
  \ref{prop:empty} the core $\Core(T \times T')$ is non-empty, and
  hence by Proposition \ref{prop:heavy2} there is an element $x \in
  F_k$ with $\omega_T(x) \in \delta_\infty$ and $\omega_{T'}(x) \in
  \delta'_\infty$.  As $f_\infty(\omega_T(x)) = \omega_{T'}(x)$, we
  see that $f_\infty(\delta_\infty) \cap \delta'_\infty \neq
  \emptyset$.

  Conversely, suppose that $f_\infty(\delta_\infty) \cap
  \delta'_\infty \neq \emptyset$.  Let $R \subset \delta$ be a ray
  whose equivalence class is mapped by $f_\infty$ into
  $\delta'_\infty$.  We can assume that $f(R) \subset \delta'$.  As
  $T/F_k$ is a finite graph, there is a point $R_0 \in R$ and elements
  $x_i \in F_k$ such that $x_iR_0 \in R$ and $d_T(R_0,x_iR_0) \to
  \infty$ as $i \to \infty$.  Now $x_if(R_0) = f(x_iR_0) \in f(R)
  \subset \delta'$ and by bounded cancellation
  $d_{T'}(f(R_0),x_if(R_0)) \to \infty$ as $i \to \infty$.  Thus the
  point $(R_0,f(R_0)) \in T \times T'$ and elements $x_i \in F_k$
  witness $\delta \times \delta'$ as a heavy quadrant.
\end{proof}

Using the above lemma we can determine which rectangles of $T \times
T'$ are in the core $\Core(T \times T')$.  This is enough to compute
the intersection number $i(T,T')$.  The following definition is
closely related to the notion of a (one-sided) \textit{cylinder} in
\cite{col:K06}.

\begin{definition}\label{def:box}
  For an oriented edge $e \subset T$, the clopen subset of $\bd T$
  consisting of equivalence classes of geodesic rays originating at
  $\iv(e)$ and containing $e$ is called a \textit{box}, which we
  denote $\clop{e}$.
\end{definition}

\begin{lemma}\label{lm:endinC}
  Let $f\co T \to T'$ be a morphism between $T,T' \in \cv$.  Given two
  edges $e \subset T$ and $e' \subset T'$, the rectangle $e \times e'
  \subset T \times T'$ is in the core $\Core(T \times T')$ if and only
  if each of the subsets $f_\infty(\clop{e}) \cap \clop{e'}$,
  $f_\infty(\clop{\bar{e}}) \cap \clop {e'}$, $f_\infty(\clop{e}) \cap
  \clop{\bar{e}'}$ and $f_\infty(\clop{\bar{e}}) \cap \clop{\bar{e}'}$
  is non-empty.
\end{lemma}

\begin{proof}
  Let $(p,p')$ be an interior point in the rectangle $e \times e'$.
  There are two directions at each of $p$ and $p'$.  These four
  directions combine to give us four quadrants $Q_i$, $i = 1,2,3,4$.
  The important feature of these four quadrants is that any quadrant
  that contains the point $(p,p')$ must also contain one of the
  $Q_{i}$'s.  Thus by Remark \ref{rm:inclusion}, the $Q_{i}$ quadrants
  are heavy if and only if $(p,p')$ is in $\Core(T \times T')$.  As
  any of the four directions at $p$ and $p'$ lie in a bounded
  neighborhood of one of $\delta_{e}, \delta_{\bar{e}}, \delta_{e'}$
  or $\delta_{\bar{e}'}$, the $Q_{i}$'s are heavy if and only if the
  quadrants $\delta_{e} \times \delta_{e'}, \delta_{e} \times
  \delta_{\bar{e}'}, \delta_{\bar{e}}\times \delta_{e'}$ and
  $\delta_{\bar{e}} \times \delta_{\bar{e}'}$ are heavy.  Therefore by
  Lemma \ref{lm:heavy}, each of the four quadrants above are heavy and
  hence $(p,p') \in \Core(T \times T')$ if and only if the sets in the
  statement of the lemma are non-empty.  As this is true for any point
  in the rectangle $e \times e'$, this rectangle is in the core if and
  only if each of these four sets is non-empty.
\end{proof}

\begin{remark}\label{rm:core}
  In a similar manner, we can also determine exactly when a vertex ($v
  \times v'$), a vertical edge ($v \times e'$), a horizontal edge ($e
  \times v'$) or a ``diagonal edge'' (in the augmented core replacing
  twice light rectangles, see \cite{ar:G05} for details) is in the
  core.  The conditions are not as simple to state for vertices or
  horizontal and vertical edges as there can be several directions at
  a vertex.  It is easy to see that $e \times e'$ is a twice-light
  rectangle if and only if $f_{\infty}(\clop{e}) = \clop{e'}$.
\end{remark}

\begin{remark}\label{rm:endinC}
  Another way to phrase Lemma \ref{lm:endinC} is given two edges $e
  \subset T$ and $e' \subset T'$ the rectangle $e \times e'$ is in the
  core $\Core(T \times T')$ if and only if there exist two geodesics
  $\rho^+,\rho^-\co(-\infty,\infty) \to T$ whose image contains $e$
  such that $f_\infty(\rho^+_\infty),f_\infty(\rho^-_{-\infty}) \in
  \clop{e'}$ and $f_\infty(\rho^+_{-\infty}),f_\infty(\rho^-_\infty)
  \in \clop{\bar{e}'}$.
\end{remark}

\begin{definition}\label{def:slice}
  Let $e$ be an edge of $T$.  The \textit{slice} of the core $\Core(T
  \times T')$ above $e$ is the set:
  \begin{equation*}\label{eq:slice}
    \Core_e = \{ e' \in T' \ | \ e \times e' \subset \Core(T \times T')
    \}.
  \end{equation*}
  Similarly define the slice $\Core_{e'} \subset T$ for an edge $e'$
  of $T'$.
\end{definition}

By Proposition \ref{prop:convex}, the slice $\Core_e$ is a subtree of
$T'$.  Clearly, every rectangle in the core belongs to one of the sets
$e \times \Core_e$.  For any interior point $p_{e} \in e$, the tree
$\{p_e\} \times \Core_e$ embeds in the quotient $\Core(T \times
T')/F_k$, as the stabilizer of any edge in $T$ is trivial. Therefore
the intersection number $i(T,T')$ can be expressed as the sum:
\begin{equation}\label{eq:slice_sum}
  i(T,T') = \sum_{e \subset T/F_k} l_{T}(e)\vol(\Core_{e})
\end{equation}
where $l_{T}(e)$ is the length of the edge $e \subset T$ and
$\vol(\Core_e)$ is the sum of the lengths of the edges in $\Core_e$.
We are therefore interested in finding the slices $\Core_e$ for a set
of representative edges for $T/F_k$.

Using the above characterization of the core we can show that the
asymptotics of $n \mapsto i(T,T'\phi^n)$ only depend on the
automorphism $\phi$ and not on the trees $T,T' \in \cv$.

\begin{lemma}\label{lm:compareslice}
  For any $T,T' \in \cv$ there are constants $K \geq 1$ and $C \geq
  0$ such that for any $T'' \in \cv$ and any edge $e \subset T''$:
 $$\frac{1}{K}\vol(\Core'_e) - C \leq \vol(\Core_e) 
 \leq K\vol(\Core'_e) + C$$ where $\Core_e,\Core'_e$ are the
 respective slices of the cores $\Core(T \times T'')$ and $\Core(T'
 \times T'')$ above the edge $e$.
\end{lemma}

\begin{proof}
  As the volume of the slices $\Core_e$ and $\Core'_e$ scale with
  proportionately with $T$ and $T'$, we are free to assume that the
  length of any edges in $T$ or $T'$ is 1.  Fix morphisms $f\co T \to
  T''$ and $g\co T' \to T$.  Then $h = f \circ g$ is a morphism from
  $T'$ to $T''$.  Fix an edge $e \subset T''$ and let $p_e$ be an
  interior point of $e$.  First we will show that there are constants
  $K_0\geq 1 $ and $C_0 \geq 0$ that only depend on the morphism $g\co
  T \to T'$ such that $\frac{1}{K_0}\vol(\Core'_e) - C_0 \leq
  \vol(\Core_e)$.

  As $g\co T' \to T$ is a quasi-isometry, there is are constants
  $K_1$ and $K_2$ such that the following hold where $\rho$ and
  $\rho'$ are geodesics in $T'$:
  \begin{itemize}
  \item[1.] If the Hausdorff distance between $\rho$ and $\rho'$ is at
    least $K_2$ then the intersection of the geodesics $[g(\rho)]$ and
    $[g(\rho')]$ in $T$ is empty; and
  \item[2.] If $\rho$ and $\rho'$ intersect in $L$ segments of length
    $K_1$ that are separated by segments of length $K_2$ then the
    intersection of the geodesics $[g(\rho)]$ and $[g(\rho')]$ in $T$
    contains at least $L$ edges.
  \end{itemize}
  For a finite subtree $X \subset T'$, let $N(X)$ be equal to the
  maximum cardinality of collection of segments of length $K_1$ in $X$
  that are pairwise distance at least $K_2$ apart.  There exist
  constants $K_0 \geq 1$ and $C_0 \geq 0$ such that
  $\frac{1}{K_0}\vol(X) - C_0 \leq N(X)$.

  Let $\mathcal{A}$ denote a maximal collection of segments of length
  $K_1$ in $\Core_e'$ that are pairwise distance at least $K_2$ apart.
  Given a segment $\alpha \in \mathcal{A}$, as $\alpha \subseteq
  \Core_e'$, by Lemma \ref{lm:endinC} (see Remark \ref{rm:endinC}) we
  can find two geodesics $\rho^+\co\R \to T'$ and $\rho^-\co \R \to
  T'$ that contain the segment $\alpha$ such that
  $h_\infty(\rho^+_\infty),h_\infty(\rho^-_{-\infty}) \in \clop{e}$
  and $h_\infty(\rho^+_{-\infty}),h_\infty(\rho^-_\infty) \in
  \clop{\bar{e}}$.  Now let $e' \subseteq T$ be an edge contained in
  the intersection of $[g(\rho^+)]$ and $[g(\rho^-)]$.  Thus
  $g_\infty(\rho^+_\infty), g_\infty(\rho^-_\infty) \in \clop{e'}$ and
  $g_\infty(\rho^+_{-\infty}), g_\infty(\rho^-_{-\infty}) \in
  \clop{e'}$.  As $h_\infty = f_\infty \circ g_\infty$ we see that
  $f_\infty(\clop{e'}) \cap \clop{e} \neq \emptyset$ and similarly for
  the three other sets in Lemma \ref{lm:endinC}, hence $e' \in
  \Core_e$.  Repeat for all other segments in $\mathcal{A}$.  By
  construction, the edges in $T$ corresponding to the $N(\Core_e')$
  segments of $\mathcal{A}$ are disjoint, therefore $N(\Core_e') \leq
  \vol(\Core_e)$.  Hence $\frac{1}{K_0}\vol(\Core'_e) - C_0 \leq
  \vol(\Core_e)$.

  A similar argument using morphisms $f'\co T' \to T''$ and $g'\co T
  \to T'$ shows that there are constants $K'_0 \geq 1 $ and $C'_0 \geq
  0$ only depending on the morphism $g'\co T \to T'$ such that
  $\frac{1}{K'_0}\vol(\Core_e) - C'_0 \leq \vol(\Core'_e)$.  This
  proves the lemma.
\end{proof}

From this lemma, we can show the independence of the asymptotics of
$i(T,T'\phi^n)$ in a special case.

\begin{corollary}\label{co:special-case}
  For any $T,T' \in \cv$, and any automorphism $\phi \in
  \Out(F_k)$, we have $i(T,T\phi^n) \sim i(T',T\phi^n)$.  
\end{corollary}

\begin{proof}
  This follows directly from \eqref{eq:slice_sum} and Lemma
  \ref{lm:compareslice} as the constants from the lemma only depend on
  $T$ and $T'$ and not on $T'' = T\phi^n$.
\end{proof}

From this special case we easily derive the full independence using
equivariance and symmetry.

\begin{proposition}\label{prop:treeindependent}
  For any $T,T',T'' \in \cv$, and any automorphism $\phi \in
  \Out(F_k)$, we have $i(T,T\phi^n) \sim i(T',T''\phi^n)$.
\end{proposition}

\begin{proof}
  Applying Corollary \ref{co:special-case} along with equivariance and
  symmetry of intersection numbers, it follows that $i(T',T\phi^n) =
  i(T'\phi^{-n},T) \sim i(T'\phi^{-n},T'') = i(T',T''\phi^n)$.
  Combining this equivalence with the equivalence in Corollary
  \ref{co:special-case} we get $i(T,T\phi^n) \sim i(T',T\phi^n) \sim
  i(T',T''\phi^n)$ as desired.
\end{proof}


\section{\texorpdfstring{Slices of the core $\mathcal{C}(T \times
    T')$}{Slices of the core C(T x T')}}\label{sc:slice}

\subsection{The map on ends for a free group
  automorphism}\label{ssc:ends}

Let $f\co T \to T'$ be a morphism between $T,T' \in \cv$.  This
descends to a Lipschitz linear homotopy equivalence $\sigma\co \Gamma
\to \Gamma'$ where $\Gamma = T/F_k$ and $\Gamma' = T'/F_k$.  Fix a
morphism $f'\co T' \to T$ such that the induced map $\sigma'\co
\Gamma' \to \Gamma$ is a homotopy inverse to $\sigma$.  Homotoping
$\sigma'$ if necessary, we assume that the image of any small open
neighborhood of any vertex of $\Gamma'$ is not contained in an edge of
$\Gamma$.  Fix basepoints $*' \in T'$ and $* \in T$ such that $f'(*')
= *$.  For simplicity, we denote the images of these basepoints in
$\Gamma$ and $\Gamma'$ by $*$ and $*'$ respectively.  We will use the
map $\sigma' \co \Gamma' \to \Gamma$ to find the map on ends
$f_\infty\co \bd T \to \bd T'$.  This is the inverse of the
homeomorphism $f'_\infty$.

Let $e$ be an (oriented) edge of $\Gamma$.  Subdivide $e$ into
$e_+e_-$ and denote the subdivision point by $p_e$.  Fix a tight edge
path $\alpha_e \subset \Gamma$ from $*$ to $p_e$ whose final edge is
$e_+$.  The path $\alpha_e$ corresponds to choosing a representative
lift for $e$ in $T$.  For simplicity, we call this lift $e$ and the
lift of $p_e$ contained in this edge $p_e$.  This lift $e$ is oriented
``away'' from $*$, i.e., $\iv(e)$ separates $p_e$ from $*$.  Consider
the set of points $\Sigma_e = (\sigma')\inv(p_e) \subset \Gamma'$.
There is one point in this set for each edge of $\Gamma'$ whose image
under $\sigma'$ crosses either $e$ or $\bar{e}$ counted with
multiplicities.

For $q \in \Sigma_e$ there is a tight path $\gamma_q \subset \Gamma'$
from $*'$ to $q$ such that up to homotopy $\alpha_e =
[\sigma'(\gamma_q)]$.  Further, the path $\gamma_q$ is unique as
$\sigma'$ is a homotopy equivalence.  If before tightening the path
$\sigma'(\gamma_q)$, the final edge is $e_+$ we assign $q$ a ``$+$''
sign.  Else, the final edge is $\bar{e}_-$ and we assign $q$ a ``$-$''
sign.

Let $\uc{\gamma}_q$ be the lift of $\gamma_q$ to $T'$ that originates
at $*'$.  Denote the terminal point of $\uc{\gamma}_q$ by $\tilde{q}$.
Let $\uc{\Sigma}_e \subset T'$ be the collection of the point
$\tilde{q}$.  In other words, $\uc{\Sigma}_e = (f')\inv(p_e)$.  This
set decomposes into $\uc{\Sigma}_e^+$ (``$+$'' points) and
$\uc{\Sigma}_e^-$ (``$-$'' points), depending on the sign of the
images of the points in $\Sigma_e$.  Since $f'$ is locally injective
on the interior of edges, for any edge $e' \subset T'$, the set
$\uc{\Sigma}_e \cap e'$ is at most a single point.  Also, since $f'$
is cellular, $\uc{\Sigma}_e$ does not contain any vertices of $T'$.

\begin{remark}\label{rm:unbounded}
  Our hypothesis that the image of any small open neighborhood of a
  vertex in $\Gamma'$ by $\sigma'$ is not contained in an edge of
  $\Gamma$ implies that every component of $T' - \uc{\Sigma}_e$ is
  unbounded.
\end{remark}

\begin{example}\label{ex:map}
  Let $T \in cv_2$ be the tree with all edge lengths equal to 1 and
  such that $\Gamma = T/F_2$ is the 2--rose whose petals are marked
  $a$ and $b$.  We use capital letters to denote the edges with
  opposite orientation. Consider the automorphism $\phi \in \Out(F_2)$
  given by $\phi(a)=ab,\phi(b)=bab$.  There is an obvious homotopy
  equivalence $\sigma'\co \Gamma \to \Gamma$ representing the
  automorphism $\phi$.  We will carefully look at the above
  construction for the map $\sigma'$.

  Subdivide the edges of $\Gamma$ as described above, creating the
  points $p_a$ and $p_b$.  Fix preferred paths $\alpha_a = a_+$ and
  $\alpha_b = b_+$.  Also subdivide $a = a_1a_2a_3$ with subdivision
  points $p_{a_+},p_{a_-}$ and $b_+ = b_1b_2, b_- = b_3b_4$ with
  subdivision points $p_{b_+},p_{b_-}$.  We may assume $\sigma'(a_1) =
  a_+, \sigma'(a_2) = a_-b_+, \sigma'(a_3) = b_-$ and $\sigma'(b_1) =
  b_+,\sigma'(b_2) = b_-a_+, \sigma'(b_3) = a_-b_+,\sigma'(b_4) =
  b_-$.

  Therefore $\Sigma_a = \{ p_{a_+},p_b \}$ and $\Sigma_b = \{
  p_{a_-},p_{b_+},p_{b_-} \}$.  Then $[\sigma'(a_1)] = [a_+] =
  \alpha_a$ and $[\sigma'(aB_4B_3)] = [abBA_-] = a_+ = \alpha_a$.
  Hence $p_{a_+}$ is a ``$+$'' point and $p_b$ is a ``$-$'' point.
  Further $[\sigma'(bAA_3)] = [babBAB_-] = b_+ = \alpha_b$,
  $[\sigma'(b_1)] = [b_+] = \alpha_b$ and $[\sigma'(bAB_4)] =
  [babBAB_-] = b_+ = \alpha_b$.  Therefore $p_{b_+}$ is a ``$+$''
  point and $p_{a_-}$, $p_{b_-}$ are ``$-$'' points.  Figure
  \ref{fig:lifts} shows the sets $\uc{\Sigma}_a$ and $\uc{\Sigma}_b$
  in the tree $T$.
  \begin{figure}[h]
    \psfrag{*}{$*$}
    \psfrag{pa}{$p_{a_+}$}
    \psfrag{pb}{$p_b$}
    \psfrag{pb+}{$p_{b_+}$}
    \psfrag{pb-}{$p_{b_-}$}
    \psfrag{pa-}{$p_{a_-}$}
    \psfrag{Sa}{$\uc{\Sigma}_a$}
    \psfrag{Sb}{$\uc{\Sigma}_b$}
    \centering
    \includegraphics{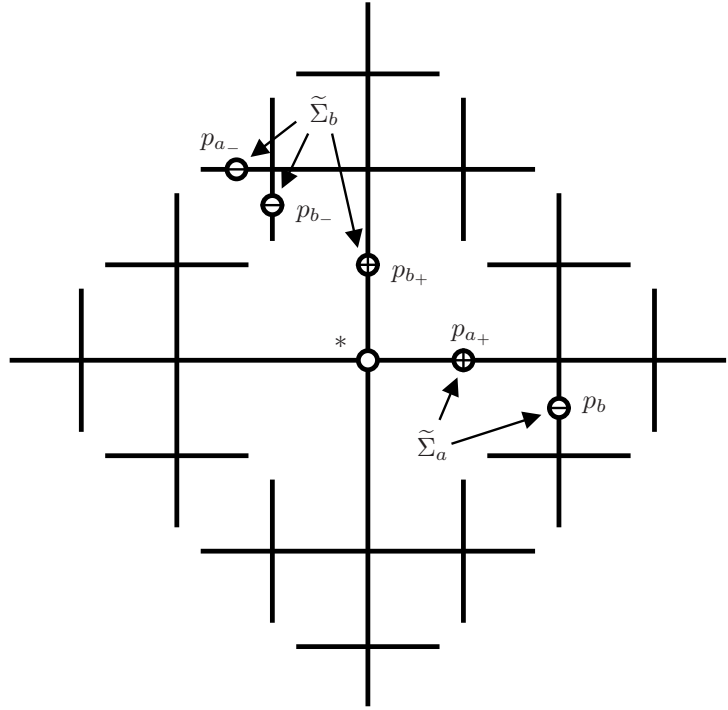}
    \caption{The sets $\uc{\Sigma}_a$ and $\uc{\Sigma}_b$ in Example
      \ref{ex:map}.  The ``$+$'' points are displayed by $\oplus$ and
      the ``$-$'' points are displayed by $\ominus$. }
    \label{fig:lifts}
  \end{figure}
\end{example}

For $\tilde{q} \in \uc{\Sigma}_e$, let $e'_{\tilde{q}}$ be the
oriented edge in $T'$ that contains $\tilde{q}$ and such that the
initial segment of $e'_{\tilde{q}}$ is contained in
$\tilde{\gamma_{q}}$, i.e., $e'_{\tilde{q}}$ is oriented away from
$*'$.  A ray containing $e'_{\tilde{q}}$ represents an end in the box
$\clop{e'_{\tilde{q}}}$.  A ray originating from $*'$ is mapped by
$f'_\infty$ into $\clop{e}$ if and only if it intersects
$\uc{\Sigma}_e$ and the final point in $\uc{\Sigma}_e$ it intersects
is a ``$+$'' point.  Likewise, a ray originating from $*'$ is mapped
by $f'_\infty$ into the complement of $\clop{\bar{e}}$ if and only if
it does not intersects $\uc{\Sigma}_e$ or if the final point of
$\uc{\Sigma_e}$ it intersects is a ``$-$'' point.  Therefore as
$f_\infty$ is the inverse of $f'_{\infty}$ we can express
$f_\infty(\clop{e})$ as a union and difference of the boxes
$\clop{e'_{\tilde q}}$ for $\tilde{q} \in \uc{\Sigma}_e$.  See
Examples \ref{ex:algorithm1} and \ref{ex:algorithm}.

A component $X \subset T' - \uc{\Sigma}_e$ is assigned a sign ``$+$''
if $f'(X)$ is contained in the direction $\delta_e \subset T$.  Else,
the image $f'(X)$ is contained in the direction $\delta_{\bar{e}}
\subset T$ and we assign it a ``$-$'' sign.  In the following lemma we
show that along any ray, the signs of the components of $T' -
\uc{\Sigma}_e$ that the ray intersects alternate.

\begin{lemma}\label{lm:alternating}
  With the above notation, let $R\co [0,\infty) \to T'$ be a ray that
  originates at a vertex of $T'$ and suppose there are two different
  components $X_0,X_1 \subset T' - \uc{\Sigma_e}$ which both have the
  same sign and for some $0 \leq t_0 < t_1$ we have $R(t_0) \in X_0$,
  $R(t_1) \in X_1$.  Then there is a component $\widehat{X} \subset T'
  - \uc{\Sigma}_e$ and $\hat{t}$ with $t_0 < \hat{t} < t_1$ such that
  $R(\hat{t}) \in \widehat{X}$ and the sign of $\widehat{X}$ is
  opposite to that of $X_0$ and $X_1$.
\end{lemma}

\begin{proof}
  Let $x_0 = R(t_0)$ and $x_1 = R(t_1)$ and consider the tight segment
  $\gamma$ connecting $x_0$ to $x_1$, this is a subsegment of the ray
  $R([0,\infty))$.  As $X_0$ and $X_1$ are different components of $T'
  - \uc{\Sigma}_e$, for some interior point $\hat{x}$ of $\gamma$
  necessarily $\hat{x} \in \uc{\Sigma_e}$.  As $f'$ maps vertices of
  $T'$ to vertices of $T$, the point $\hat{x}$ is in the interior of
  some edge $\hat{e} \subset T'$.  As $f'$ is injective on the
  interior of edges, there are points $\hat{x}_+,\hat{x}_-$ in
  $\hat{e}$ close to $\hat{x}$ that are in components of $T' -
  \uc{\Sigma}_e$ with opposite sign.  One of these points gives
  $R(\hat{t})$.
\end{proof}

Let $T'_e \subset T'$ be the subtree spanned by the points
$\uc{\Sigma_e}$.  If $e'$ is an edge of $T'$ that is not contained in
$T'_e$ and is oriented away from $T'_e$, then $f'_\infty(\clop{e'})$
is contained in either $\clop{e}$ or $\clop{\bar{e}}$.  Thus either
$f'_\infty(\clop{e'}) \cap \clop{e}$ or $f'_\infty(\clop{e'}) \cap
\clop{\bar{e}}$ is empty and therefore by Lemma \ref{lm:endinC}, the
edge $e'$ is not in the slice $\Core_e$.  Hence, the slice $\Core_e$
is contained in the subtree of interior edges of $T'_e$.  We will show
that the difference between the volume of $T'_e$ and $\Core_e$ is
bounded, at least when then morphism $f'\co T' \to T$ is the lift of a
train-track representative.

First, let's consider a situation where an interior edge of $T'_e$ is
not in $\Core_e$.  Without loss of generality, this implies that any
geodesic ray which contains this edge is eventually contained in a
component of $T' - \uc{\Sigma}_e$ that has a ``$+$'' sign.  This can
happen if there is a vertex of $T'_e$ such that all but one of its
adjacent edges contains a point in $\uc{\Sigma}_e$ which are a
terminal vertices of $T'_e$.  In this case we would like to remove a
neighborhood of this vertex and \emph{consolidate} the set of points,
repeating if necessary.  See Figure \ref{fig:consolidation}.

\begin{figure}[ht]
  \centering
  \includegraphics{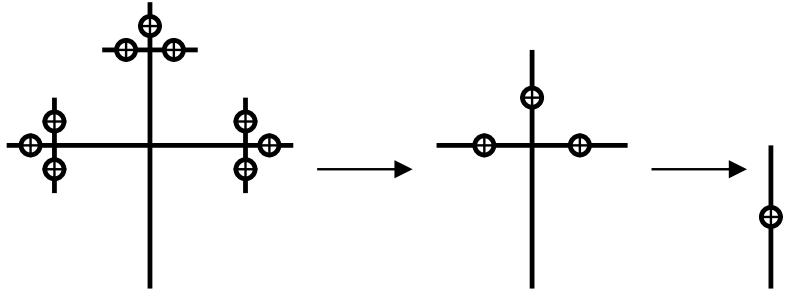}
  \caption{Consolidating vertices in $T'_e$.}
  \label{fig:consolidation}
\end{figure}

We now define this operation in detail and show that it terminates in
the slice of the core $\Core_e$.  Let $V_t(T'_e)$ denote the set of
terminal vertices of $T'_e$.  A vertex $v \in T'_e$ is \textit{full}
if the valence of $v$ in $T'_e$ equals the valence of $v$ in $T'$.

\begin{definition}\label{def:removable}
  Let $v \in T'_e$ be a full vertex with adjacent edges
  $\hat{e},e_1,\ldots,e_k$ and suppose $\tv(e_i) \in V_t(T'_e)$ for $i
  = 1,\ldots,k$ but $\tv(\hat{e}) \notin V_t(T'_e)$.  Then $v$ is a
  \textit{removable} vertex if $\{\tv(e_1),\ldots,\tv(e_k)\} \subset
  \uc{\Sigma}_e$.
\end{definition}

At a removable vertex $v$, subdivide the edge $\hat{e}$ into
$\hat{e}_0\hat{e}_1$ and denote the subdivision vertex by
$p_{\hat{e}}$.  Since the image of any small open neighborhood of $v$
is not contained in a single edge of $T$, the edge $\hat{e}$ does not
contain any points in $\uc{\Sigma}_e$.  Remove the $(k+1)$--pod
$\hat{e}_0,e_1,\ldots,e_k$, leaving the vertex $p_{\hat{e}}$.  Delete
the points $\{\tv(e_1),\ldots,\tv(e_k)\}$ from $\uc{\Sigma}_e$ and add
the point $p_{\hat{e}}$.  In this new subtree $p_{\hat{e}}$ is a
terminal vertex.  This changes the components of $T' - \uc{\Sigma}_e$.
The components of $T' - \{\tv(e_i) \}$ which do not contain $v$ are
combined with the $(k+1)$--pod $\hat{e}_0,e_1,\ldots,e_k$.  The
signs of all of these components were the same before the operation,
and we assign this sign to the new component.

This process may create removable vertices, specifically the vertex
$\tv(\hat{e})$.  Notice that for this new collection of points
$\uc{\Sigma}_e$ and components of $T' - \uc{\Sigma}_e$, Remark
\ref{rm:unbounded} and Lemma \ref{lm:alternating} still hold.

Repeat this process until there are no removable vertices.  As $T'_e$
is a finite subtree and at each step we remove a finite subtree that
does not disconnect the new space, this process will terminate with a
subtree $Y'_e$.  In the following lemma we collect some elementary
properties of the subtrees that are removed.

\begin{lemma}\label{lm:subtreeA}
  With the above notation, suppose that $A$ is a connected component
  of $T'_e - Y'_e$.  Then
  \begin{itemize}
  \item[1.] all but one of the terminal vertices of $A$ is a terminal
    vertex in $T'_e$ and belongs to the set $\uc{\Sigma}_e$; and

  \item[2.] $\uc{\Sigma}_e \cap A \subset V_t(A)$.
  \end{itemize}
\end{lemma}

\begin{proof}
  If $A$ has two terminal vertices that are not terminal vertices of
  $T'_e$, then $Y'_e$ is not connected, which is a contradiction.  As
  all of the terminal vertices of $T'_e$ belong to $\uc{\Sigma}_e$,
  1. holds.

  Suppose that $e' \subset A$ contains a point of $\uc{\Sigma}_e$ that
  is not a terminal vertex of $T'_e$.  Orient $e'$ to point away from
  the terminal vertex of $A$ that is not a terminal vertex of $T'_e$.
  Then the component of $T'_e - \uc{\Sigma}_e$ which contains the
  vertex $\tv(e')$ is bounded, which is a contradiction, hence
  2. holds.
\end{proof}

Let $Z'_e$ denote the subtree of edges in $Y'_e$ not adjacent to
valence one vertices, i.e., the edges of $Y'_e$ that are edges in
$T'$.

\begin{lemma}\label{lm:signedrays}
  For any oriented edge $e' \subset Z'_e$ there are rays $R^+\co
  [0,\infty) \to T'$ and $R^-\co[0,\infty) \to T'$ containing $e'$
  with the specified orientation, such that $f'_\infty(R^+_\infty) \in
  \clop{e}$ and $f'_\infty(R^-_\infty) \in \clop{\bar{e}}$.
\end{lemma}

\begin{proof}
  Let $e' \subset T'_e$ and assume that for every ray $R$ containing
  $e'$ we have $f'_\infty(R_\infty) \in \clop{\bar{e}}$.  Let $A$ be
  the component of $T'_e - \iv(e')$ that contains $e'$.  In
  particular, components of $T' - \uc{\Sigma}_e$ contained in
  $\clop{e'}$ that are adjacent to $A$ have sign ``$-$''.  We will
  show that $A$ is removed from $T'_e$ in the construction of $Z'_e$
  and hence $e'$ is not in $Z'_e$.
 
  Let $X$ be the component of $T' - \uc{\Sigma}_e$ that contains
  $\tv(e')$ and $X_0$ the component of $X - \{ \iv(e) \}$ that
  contains $e'$.  We claim that $A = X_0$.

  If there is a point in $A$ that is not in $X_0$, then $A$ contains a
  point in $\uc{\Sigma}_e$ that is not a terminal vertex of $A$.  Let
  $x$ be such a point that is closest to a terminal vertex of $A$ and
  $X'$ the component of $T' - \uc{\Sigma}_e$ adjacent to both $x$ and
  a terminal vertex of $A$.  As this component is unbounded, we can
  find a ray that contains $e'$ and is eventually contained in $X'$.
  By hypotheses, the sign of this component is ``$+$'' (as it is
  adjacent to a component with a ``$-$'' sign), and hence there is a
  ray $R$ containing $e'$ such that $f'_\infty(R_\infty) \in
  \clop{e}$, which contradicts our assumptions.

  If $X_0$ is bounded then every ray containing $e'$ intersects some
  point in $\uc{\Sigma}_e$.  Hence, if there is a point in $X_0$ that
  is not in $A$ then $X_0$ is unbounded and therefore $X$ has sign
  ``$-$''.  But then $X_0$ is also adjacent to a component of
  $T'-\uc{\Sigma}_e$ with a ``$+$'' sign and as before there is a ray
  $R$ containing $e'$ such that $f'_\infty(R_\infty) \in \clop{e}$,
  which contradicts our assumptions.

  Therefore as vertices of $X_0$ that are vertices of $T'$ are full,
  vertices of $A$ adjacent to terminal vertices of $A$ are removable.
  We can repeat to see that $A$ is removed by removing removable
  vertices and edges adjacent to valence one vertices.

  Hence, if $e' \subset Z'_e$, not every ray containing $e'$ is mapped
  by $f'_\infty$ to $\clop{e'}$.  Similarly, we see that for $e'
  \subset Z'_e$, not every ray that contains $e'$ is mapped by
  $f'_\infty$ to $\clop{e}$.
\end{proof}

\begin{lemma}\label{lm:slice}
  For any edge $e \subset T$ the slice $\Core_e$ of the core $\Core(T
  \times T')$ is the subtree $Z'_e \subset T'$.
\end{lemma}

\begin{proof}
  We have already seen that $\Core_e \subseteq T'_e$.  By Lemma
  \ref{lm:subtreeA} we see that for any edge $e' \subset T'_e - Z'_e$
  one of the sets $f_\infty(\clop{e'})$ or $f_\infty(\clop{\bar{e}'})$
  is contained in either $\clop{e}$ or $\clop{\bar{e}}$.  Therefore by
  Lemma \ref{lm:endinC} we have $\Core_e \subseteq Z'_e$.  By Lemmas
  \ref{lm:endinC} and \ref{lm:signedrays} we have $Z'_e \subseteq
  \Core_e$.  Hence $Z'_e = \Core_e$.
\end{proof}


\subsection{\texorpdfstring{Examples of $\Core(T \times T')$}{Examples
    of C(T x T')}}\label{ssc:examples}

We now present some examples of computing the map on ends and building
the core for some trees $T,T' \in \cv$.

\begin{notation}\label{no:box}
  We adopt the convention of using ``$\clop{e} + \clop{e'}$'' to
  denote the union of the two boxes $\clop{e},\clop{e'}$.  Also, when
  $\clop{e'}$ is contained within $\clop{e}$ we use ``$\clop{e} -
  \clop{e'}$'' to denote the set of ends contained within $\clop{e}$
  but not $\clop{e'}$.
\end{notation}

\begin{example}\label{ex:algorithm1}
  Let $T \in cv_2$ and $\phi \in \Aut(F_2)$ be as in Example
  \ref{ex:map}.  We can identify $T$ with the Cayley graph for $F_2$
  and the ends of $T$ with right-infinite words in $F_2$.  From Figure
  \ref{fig:lifts} we see the map on ends $f_\infty\co \bd T\phi \to
  \bd T$ is:
  \begin{align*}
    f_\infty\co \clop{a} \mapsto & \clop{a} - \clop{aB} \\
    \clop{b} \mapsto & \clop{b} - \clop{bAB} - \clop{bAA}
  \end{align*}
  In terms of right-infinite words, $f_\infty(\clop{a}) = \clop{a} -
  \clop{aB}$ is interpreted as saying that the right-infinite words
  starting with $a$ are taken homeomorphically by $\phi\inv$ to the
  set of right-infinite words starting with $a$ except those that
  start with $aB$.

  Also, from Figure \ref{fig:lifts}, we see that the subtrees $T'_a$
  and $T'_b$ do not have any removable vertices, therefore the slices
  $\Core_a,\Core_b$ of the core $\Core(T\phi \times T)$ are the
  respective subtrees with the ``half'' edges removed. Thus the slice
  $\Core_b$ is the edge $bA$ and the slice $\Core_a$ is the vertex
  $a\tilde{*}$.  This is not the complete description of the core in
  this case, there are some vertical edges ($v \times e$) that can be
  found by examining a homotopy inverse to $\sigma'$.  This does
  however give us all of the 2--cells in the core, hence we see
  $i(T\phi,T) = 1$.
\end{example}

\begin{example}\label{ex:algorithm}
  Let $T \in cv_3$ be the tree with all edge lengths equal to 1 and
  such that $\Gamma = T/F_3$ is the 3--rose whose petals are labeled
  $a,b,c$.  In this example we will use the above algorithm to find
  the map on ends $f_\infty \co \bd T \to \bd T\phi$ and build the
  slices of the core $\Core(T \times T\phi)$ for the automorphism
  $\phi \in \Aut(F_{3})$ given by $\phi(a)=baC,\phi(b)=cA,\phi(c)=a$
  where $F_{3} = \langle a,b,c \rangle$, here we use capital letters
  to denote the inverses of the generators.  The inverse of $\phi$ is
  the map $\phi\inv(a)=c,\phi\inv(b)=ab,\phi\inv(c)=bc$.  There is an
  obvious homotopy equivalence $\sigma'\co \Gamma \to \Gamma$
  representing $\phi\inv$.  Subdivide the edges of $\Gamma$ as in
  Example \ref{ex:map}.  We can assume that $\sigma'(p_a) = p_c$.
  Further subdivide $b_+ = b_1b_2$ and $b_- = b_3b_4$ with subdivision
  points $p_{b_+},p_{b_-}$ such that $\sigma'(p_{b_+}) = p_a$ and
  $\sigma'(p_{b_-}) = p_b$.  Likewise subdivide both $c_+ = c_1c_2$
  and $c_- = c_3c_4$ with subdivision points $p_{c_+},p_{c_-}$ such
  that $\sigma'(p_{c_+}) = p_b,\sigma'(p_{c_-}) = p_c$.  Our preferred
  paths are $\alpha_a = a_+,\alpha_b = b_+$ and $\alpha_c = c_+$.

  The preimages of $p_a,p_b,p_c$ are $\Sigma_a = \{p_{b_+} \}$,
  $\Sigma_b =\{p_{b_-},p_{c_+}\}$ and $\Sigma_c =
  \{p_{a},p_{c_-}\}$.  As $[\sigma'(b_1)] = [a_+] = \alpha_a$ we see
  that $f_{\infty}(\clop{a}) = \clop{b}$.  Then $[\sigma'(cAB_4)] =
  [bcCB_-] = b_+ = \alpha_b$ and $[\sigma'(c_1)] = [b_+] = \alpha_b$,
  hence $f_\infty(\clop{b}) = \clop{c} - \clop{cAB}$.  Finally we see
  $[\sigma'(a_+)] = [c_+] = \alpha_c$ and $[\sigma(aC_4)] = [cC_-] =
  c_+ = \alpha_c$ and hence $f_\infty(\clop{c}) = \clop{a} -
  \clop{aC}$.

  $F_k$--equivariance can be used to find the image of any other box,
  for example we will compute $f_{\infty}(\clop{B})$.  First notice
  that $\clop{B} = B(\neg\clop{b})$ where $\neg\clop{b}$ denotes the
  complement of $\clop{b}$ in $\bd T$.  Therefore $f_\infty(\clop{B})
  = f_\infty(B(\neg\clop{b})) = \phi(B)f_\infty(\neg\clop{b}) =
  aC(\clop{a} + \clop{A} + \clop{b} + \clop{B} + \clop{cAB} +
  \clop{C}) = \clop{aC} + \clop{B}$.

  From $f_\infty$ we see that the slice $\Core_a$ is empty ($a \times
  c$ is a twice-light rectangle), $\Core_c$ is a single vertex and
  $\Core_b$ is a single edge $cA$.  Hence $i(T,T\phi) = 1$.

  To get a more complicated example of the core, we look at $\Core(T
  \times T\phi^3)$.  Explicitly, $\phi^{3}$ is the automorphism:
  \begin{equation*}
    \begin{array}{rcl} 
      a & \mapsto & acABcAbaCAcAB \\
      b & \mapsto & baCacABaC \\
      c & \mapsto & cAbaCA.
    \end{array}    
  \end{equation*}
  The third power of the above map $f_\infty$ is the map on ends for
  $\phi^{3}$.  To find this map, we can either use the algorithm with
  the automorphism $\phi^{-3}$, or we can formally iterate
  $f_\infty$.  For example, to find $f^2_\infty(\clop{b}) =
  f_\infty(\clop{c} - \clop{cAB}) = f_\infty(\clop{c})
  -\phi(cA)f_\infty(\clop{B}) = \clop{a} - \clop{aC} -
  acAB(\clop{aC} + \clop{B}) = \clop{a} - \clop{aC} - \clop{acaBaC} -
  \clop{acABB}$.  From either procedure we find:
  \begin{align*}
    \clop{a} \mapsto & \clop{a} - \clop{aC} - \clop{acABB} -
    \clop{acABaC} \\
    f^3_\infty\co \clop{b} \mapsto & \clop{b} - \clop{baCA} -
    \clop{baCC} -
    \clop{baCacABaCB} - \clop{baCacABaCaC} \\
    & - \clop{baCacABaCbaCC} -
    \clop{baCacABaCbaCA} \\
    \clop{c} \mapsto & \clop{c} - \clop{cAB} - \clop{cAbaCAC} -
    \clop{cAbaCAA} - \clop{cAbaCAcAB}
  \end{align*}
  Using this map and Lemma \ref{lm:slice} we can build the slices of
  the core $\Core(T \times T\phi^3)$.  The sets $e \times \Core_{e}$
  for $e = a,b,c$ are displayed in Figure \ref{fg:corephi-6}.  The
  labeling of the edges and vertices is as follows: for each set $e
  \times \Core_e$, the bottom left vertex is shown.  For instance for
  the set $a \times \Core_a$, this vertex is $(*,a)$ where $*$ denotes
  the basepoint and $a$ denotes the image of the basepoint by $a$.
  The rest of the edges are labeled by their image in the rose
  $T/F_3$.  Hence the two vertices adjacent to $(*,a)$ are $(a,a)$ and
  $(*,ac)$.  Reading upward, the remaining vertices above $(*,a)$ are
  $(*,ac)$, $(*,acA)$, $(*,acAB)$ and $(*,acABa)$.  Looking at the
  image for $f_\infty(\clop{a})$, we see that the span of these
  vertices is the slice $\Core_a$.  Adding up the number of squares we
  see $i(T,T\phi^3) = 23$.

  The colors and arrows denote the identifications, the thick black
  lines are free edges.  As an example of the identifications let's
  look at the vertex on the bottom left, this vertex is $(*,a)$.  Then
  $c(*,a) = (c,\phi^{3}(c)a) = (c,cAbaC)$, which is the fifth vertex
  from the bottom on the right.  The identifications for the other
  vertices are found similarly.  The observant reader will note that
  $\chi(\Core(T \times T\phi^3)) = -2$ as expected.

  \begin{figure}[p]
    \psfrag{a}{{$a$}} \psfrag{b}{{$b$}} \psfrag{c}{{$c$}}
    \psfrag{A}{{$A$}} \psfrag{B}{{$B$}} \psfrag{C}{{$C$}}
    \psfrag{cA}{{$cA$}} \psfrag{ac}{{$ac$}} \psfrag{ba}{{$ba$}}
    \includegraphics[width=4in]{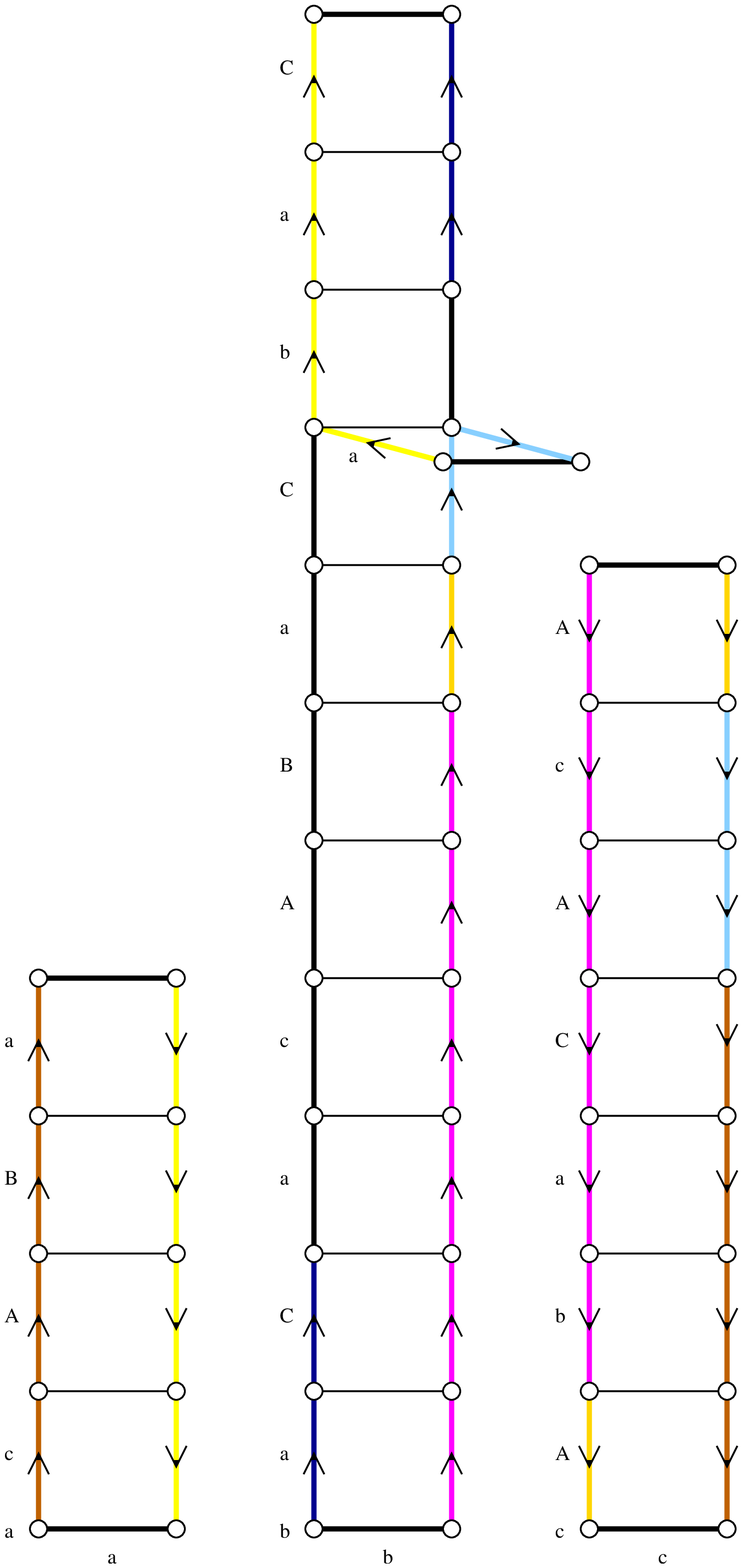}
    \caption{The core $\Core(T \times T\phi^{3})$ from Example
      \ref{ex:algorithm}.}\label{fg:corephi-6}
  \end{figure}
\end{example}


\section{Bounding consolidation}\label{sc:consolidation}

When the morphism $f\co T' \to T$ used in Section \ref{sc:slice} is
the lift of a train-track representative, we are able to control the
difference in volume between $T'_e$ and $\Core_e$ by bounding the
amount of consolidation that occurs when removing removable
vertices. Any train-track map satisfies the hypothesis used in Section
\ref{ssc:ends} that the image of a small open neighborhood of a vertex
is not contained in an edge.  We begin by showing that there is a
bound on the depth of consolidation.

\begin{lemma}\label{lm:boundedvol}
  Let $\sigma\co \Gamma \to \Gamma$ be a train-track map for a fully
  irreducible automorphism $\phi \in \Out(F_k)$ and $f\co T \to T$ a
  lift.  Let $e$ be an edge of $T$, $p_e$ an interior point of $e$,
  $T^n_e $ be the subtree spanned by $(f^n)\inv(p_e)$ and $Y^n_e$ the
  subtree found by iteratively removing removable vertices.  There is
  a constant $C \geq 0$, independent of $n$, such that any subtree $A$
  of $T^n_e$ that is removed in the formation of $Y^n_e$ has volume
  less than $C$.
\end{lemma}

\begin{proof}
  Without loss of generality, we can assume that $A$ is a maximal
  subtree that is removed.  We will show that the diameter of $A$ is
  bounded; since the length of an edge in $T$ is bounded and the
  valence of any vertex in $T$ is bounded, this implies that the
  volume of $A$ is bounded.

  For any interior vertex $v \in A$, label the edges incident to $v$
  by $\hat{e},e_1,\ldots,e_k$ where $\hat{e}$ is the edge that leads
  to the unique terminal vertex of $A$ that is not a terminal vertex
  of $T^n_e$.  Any geodesic ray originating at $v$ and containing one
  of the $e_i$ contains exactly one point from $(f^n)\inv(p_e)$.

  We claim that the only legal turns are contained in the edge paths
  $\bar{e}_i\hat{e}$ for $i = 1,\ldots,k$.  A turn contained in
  $\bar{e}_ie_j$ is illegal since $f^n$ identifies the germs of the
  edges $e_i$ and $e_j$.  Since every edge must be in some legal turn
  we see that the turns in $\bar{e}_i\hat{e}$ are legal.

  Therefore, the geodesic containing any two terminal vertices of $A$
  has a single illegal turn.  Write this geodesic as $\beta \cdot
  \gamma$ where $\beta$, $\gamma$ are legal paths.  If the length of
  $\beta \cdot \gamma$ is more than twice the critical constant for
  $f$, then at least one of $\beta$ or $\gamma$ has length bounded
  below by the critical constant.  Therefore the length of $[f^m(\beta
  \cdot \gamma)]$ goes to infinity.  But since $[f^m(\beta \cdot
  \gamma)]$ is a point for $m \geq n$, we see that the diameter of $A$
  is bounded by twice the critical constant.
\end{proof}

\begin{proposition}\label{prop:comparevol}
  Let $\sigma\co \Gamma \to \Gamma$ be a train-track map for a fully
  irreducible automorphism $\phi \in \Out(F_k)$ and $f\co T \to T$ a
  lift.  Let $e$ be an edge of $T$, $p_e$ an interior point of $e$,
  $T^n_e$ the subtree spanned by $(f^n)\inv(p_e)$ and $\Core^n_e$ the
  slice of the core $\Core(T \times T\phi^n)$ above $e$.  Then there
  exist constants $K \geq 1$ and $C \geq 0$ such that for any $n \geq
  0$:
  \[\vol(\Core^n_e) \leq \vol(T^n_e) \leq K\vol(\Core^n_e) + C.\]
\end{proposition}

\begin{proof}
  As before, denote by $Y^n_e$ the subtree of $T^n_e$ obtained by
  iteratively removing removable vertices and $Z^n_e$ the subtree of
  interior edges of $Y^n_e$.  Then by Lemma \ref{lm:endinC} we have
  that $Z^n_e = \Core^n_e$. Let $M$ and $m$ denote the length of the
  longest and shortest edges of $\Gamma$ respectively and let $b$
  denote the maximum valence of any vertex in $\Gamma$.  Then by
  adding at most $b$ edges of length $M$ to each vertex of $Z^n_e$, we
  can cover $Y^n_e$.  Since $Z^n_e$ has at most $\frac{\vol(Z^n_e)}{m}
  + 1$ vertices, this shows that there are constants $K_1 \geq 1$ and
  $C_1 \geq 0$ such that:
  \[\vol(\Core^n_e) \leq \vol(Y^n_e) \leq K_1\vol(\Core^n_e) + C_1.\]
  Therefore we only need to show that we can find constants $K_2,C_2$
  such that $\vol(T^n_e) \leq K_2\vol(Y^n_e) + C_2$.

  By Lemma \ref{lm:boundedvol}, $Y^n_e$ is obtained from $T^n_e$ by
  removing disjoint subtrees, all of whose volumes are bounded by a
  constant $K_3$.  Since the maximum number of these trees is
  $b(\frac{\vol(Y^n_e)}{m} + 1)$ we see that:
  $$\vol(T^n_e) \leq \vol(Y^n_e) + K_3b\left(\frac{\vol(Y^n_e)}{m} + 1\right) =
  \left(\frac{K_3b}{m} + 1\right)\vol(Y^n_e) + K_3b.$$ Hence the
  proposition follows.
\end{proof}

Combining the above results, we get a way to estimate the growth rate
of the intersection number for a fully irreducible automorphism.

\begin{proposition}\label{prop:estimate}
  Let $\phi$ be a fully irreducible automorphism and $f\co T \to T$ a
  lift of a train-track representative for $\phi$.  Then for any edge
  $e \subset T$ and trees $T',T'' \in \cv$:
  \[ i(T',T''\phi^n) \sim \vol(T^n_e) \] where $p_e$ is the midpoint
  of $e$ and $T^n_e \subset T$ is the subtree spanned by the points in
  $(f^n)\inv(p_e)$.
\end{proposition}

\begin{proof}
  By Proposition \ref{prop:comparevol}, we have $\vol(T^n_e) \sim
  \vol(\Core^n_e)$ where $\Core^n_e$ is the slice of $\Core(T \times
  T\phi^n)$ above $e$.  Since $\phi$ is fully irreducible,
  $\vol(T^n_e) \sim \vol(T^n_{e'})$ for any edges $e,e' \subset T$.
  (This becomes clear in Section \ref{sc:growth} where we compute
  $\vol(T^n_e)$ up to $\sim$ equivalence independent of the edge $e$.)
  Thus $\vol(T^n_e) \sim \sum_{e' \subset T/F_k}
  l_T(e')\vol(\Core^n_{e'}) = i(T,T\phi^n) \sim i(T',T''\phi^n)$.  The
  final equivalence is from Proposition \ref{prop:treeindependent}.
\end{proof}

Using the techniques developed so far, we can get a lower bound on the
asymptotics of $n \mapsto i(T,T'\phi^n)$.

\begin{corollary}\label{co:lowerbound}
  Suppose $\phi$ is a fully irreducible automorphism and let $\lambda$
  and $\lambdap$ denote the growth rates of $\phi$ and $\phi\inv$
  respectively.  Then for any $T,T' \in \cv$, there exist constants $K
  \geq 1$ and $C \geq 0$ such that
  $\frac{1}{K}\max\{\lambda,\lambdap\}^n - C \leq i(T,T'\phi^n)$.
\end{corollary}

\begin{proof}
  By Proposition \ref{prop:estimate} we only need to find a lower
  bound for $\vol(T^n_e)$ where $T^n_e$ is the span of
  $(f^n)\inv(p_e)$ for a lift of a train-track map $f \co T \to T$.
  Since $\phi$ is fully irreducible and $f$ uniformly expands edges by
  $\lambda$, the cardinality of $(f^n)\inv(p_e)$ is $\sim \lambda^n$.
  Since each point of $(f^n)\inv(p_e)$ lies in a unique edge of
  $T^n_e$, there exists constants $K \geq 1$ and $C \geq 0$ such that
  $\frac{1}{K}\lambda^n - C \leq i(T,T\phi^n)$.  Repeating this
  argument for $f'\co T' \to T'$, a lift of a train-track map
  representing $\phi\inv$, we see that there are constants $K' \geq 1$
  and $C' \geq 0$ such that $\frac{1}{K'}\lambdap^n - C' \leq
  i(T',T'\phi^{-n})$.  Since $i(T',T'\phi^{n}) \sim i(T,T\phi^{-n})$,
  the result follows.
\end{proof}


\section{Growth Rates}\label{sc:growth}

In this section we compute the growth rate of $n \mapsto \vol(T^n_e)$
where $T^n_e$ is as in Proposition \ref{prop:estimate}.  There are two
cases depending on whether or not $T^+$, the stable tree of $\phi$, is
geometric.  Every fully irreducible automorphism admits a train-track
representative (called {\it stable}) with at most one orbit of
indivisible Nielsen paths (this was proved in
\cite[Lemma~3.2]{ar:BH92} in the case of period 1 but the proof in
general is identical).

Moreover the existence of an indivisible Nielsen path characterizes
whether $T^+$ is geometric.

\begin{theorem}[Theorem 3.2 \cite{un:BF}]\label{th:geo-nongeo}
  Let $\phi \in \Out(F_k)$ be a fully irreducible automorphism,
  $\sigma\co \Gamma \to \Gamma$ a stable train-track map and $T^+$ the
  stable tree for $\phi$.  Then $T^+$ is geometric if and only if
  $\sigma\co\Gamma\to\Gamma$ contains an indivisible orbit of Nielsen
  paths.
\end{theorem}

We will not make use any of the additional properties that stableness
of a train-track map guarantees with the exception of Proposition
\ref{prop:parageo-map}.  These properties are mentioned within this
proposition.  See \cite{ar:BH92} for a definition.  The following
definition is of central importance to understanding the tree $T^n_e$
from Proposition \ref{prop:estimate}.

\begin{definition}\label{def:i-vp}
  An \textit{i--vanishing path for $\sigma\co \Gamma \to \Gamma$} ($i
  \geq 0$) is an immersion $\iota\co [0,1] \to \Gamma$ such that the
  image $\sigma^{i}\iota([0,1])$ is homotopic to a point relative to
  the endpoints.  We will always assume that $i$ is minimal.  If $f\co
  T \to T$ is a lift of $\sigma\co \Gamma \to \Gamma$, a
  \textit{i--vanishing path} is an embedding $\iota\co [0,1] \to T$
  such that $f^i\iota(0) = f^i\iota(1)$.  Clearly any $i$--vanishing
  path in $T$ projects to a $i$--vanishing path in $\Gamma$ and vice
  versa.  A \emph{vanishing path} is an $i$--vanishing path for some
  $i$.
\end{definition}

The importance of vanishing paths is given by the following remark.

\begin{remark}\label{rm:union-i-vp}
  Suppose $\sigma\co \Gamma \to \Gamma$ is a map and $f\co T \to T$ is
  a lift.  Let $T^n_e$ be the subtree spanned by $(f^n)\inv(p_e)$ for
  any edge $e \subset T$.  Then $T^n_e$ can be expressed as a union of
  $i$--vanishing paths where $i \leq n$.  Infact, the path joining any
  two points in $(f^n)\inv(p_e)$ is an $i$--vanishing path for some $i
  \leq n$.
\end{remark}

Therefore we are interested in examining the lengths of vanishing path.
As we demonstrate in Section \ref{ssc:geo}, every indivisible Nielsen
path $\gamma$ contains subpaths $\gamma_\epsilon$ that are
$i$--vanishing, for arbitrarily large $i$.  Note that all such
vanishing paths have uniformly bounded length.  In Section
\ref{ssc:nongeo} we will prove a converse to this observation in the
absence of indivisible Nielsen paths. Specifically, we show that:
\begin{itemize}
\item[1.] If $T^+$ is geometric then every vanishing path is a
  composition of suitable $\gamma_\epsilon$'s. See Proposition
  \ref{prop:geo-vp}.
\item[2.] If $T^+$ is nongeometric then there exist arbitrarily long
  vanishing paths that are not compositions of shorter vanishing
  paths. In fact, any $i$--vanishing path has length approximately
  $\lambdap^i$, where $\lambdap$ is the expansion factor for
  $\phi\inv$. See Proposition \ref{prop:ngub}.
\end{itemize}

Example \ref{ex:geom-nongeom} shows these different behaviors.  We
begin by examining the case when $T^+$ is nongeometric.


\subsection{\texorpdfstring{$T^+$ nongeometric}{T+
    nongeometric}}\label{ssc:nongeo}

\begin{convention}\label{con:ng}
  We fix some notation for use in the rest of this section: $\phi$ is
  a fully irreducible automorphism, $\sigma\co \Gamma \to \Gamma$ is a
  train-track map for $\phi$ with expansion factor $\lambda$,
  $\sigma'\co \Gamma' \to \Gamma'$ is a train-track map for $\phi\inv$
  with expansion factor $\lambdap$ and $\tau\co \Gamma \to \Gamma'$,
  $\tau'\co \Gamma' \to \Gamma$ are Lipschitz homotopy equivalences
  representing the change in marking.  This is summarized in the
  following commutative (up to homotopy) diagram:
  \begin{equation}\label{eq:diagram}
    \xymatrix{
      \cdots \ar[r]^{\sigma} & \Gamma \ar[r]^{\sigma} \ar@/_/ [d]_{\tau} & 
      \Gamma \ar[r]^\sigma \ar@/_/ [d]_{\tau} & 
      \Gamma \ar[r]^{\sigma} \ar@/_/ [d]_{\tau} & \cdots \\
      \cdots &  \ar[l]_{\sigma'} \Gamma' \ar@/_/ [u]_{\tau'} &  
      \Gamma' \ar[l]_{\sigma'} \ar@/_/ [u]_{\tau'} & 
      \Gamma' \ar[l]_{\sigma'} \ar@/_/ [u]_{\tau'} & \cdots \ar[l]_{\sigma'}}
  \end{equation}
\end{convention}

As stated above, we will show that when $\Gamma$ does not contain an
indivisible orbit of periodic Nielsen paths that the length of an
$i$--vanishing path for $\sigma$ is approximately $\lambdap^i$.  First
we give an upper bound on the length of an $i$--vanishing path.  The
following proposition does not depend on the absence of an indivisible
orbit of periodic Nielsen paths.

\begin{proposition}\label{prop:lb}
  Let $\sigma\co\Gamma\to \Gamma$ be as in Convention
  \ref{con:ng}. There exists a constant $K \geq 0$ such that for every
  $i$--vanishing path $\gamma$ for $\sigma\co\Gamma \to \Gamma$ we
  have $length(\gamma) \leq K\lambdap^i$.
\end{proposition}

\begin{proof}
  Since the composition $\sigma'\tau\sigma$ is homotopic to $\tau$ and
  the composition $\tau'\tau$ is homotopic to the identity map, there
  are constants $K_1,K_2\geq 1$, such that any path $\gamma$ in
  $\Gamma$ satisfies
  \begin{align}
    |length([\tau(\gamma)]) -
    length([\sigma'\tau\sigma(\gamma)])|&\leq K_1\mbox{ and}
    \label{eqn:homotopy2} \\
    |length(\gamma) - length([\tau'\tau(\gamma)])|&\leq
    K_2.\label{eqn:homotopy1}
  \end{align}
  For a $1$--vanishing path $\gamma$, using the inequality
  \eqref{eqn:homotopy2} we have $$length([\tau(\gamma)]) \leq
  length([\sigma'\tau\sigma(\gamma)]) + K_1 = K_1$$ as
  $length([\sigma(\gamma)]) = 0$.

  We now proceed by induction.  Suppose that the image under $\tau$ of
  any $(i-1)$--vanishing path has length at most
  $K_1\sum_{j=1}^{i-1}\lambdap^{j-1}$; note that the previous
  paragraph verified the base case, a $1$--vanishing path.  Now
  consider an $i$--vanishing path $\gamma$.  Then $\sigma(\gamma)$ is
  an $(i-1)$--vanishing path, and thus the inductive hypothesis yields
  $length([\tau\sigma(\gamma)]) \leq K_1\sum_{j=1}^{i-1}
  \lambdap^{j-1}$.  Since $\sigma'$ is a train-track map this implies
  $length([\sigma'\tau\sigma(\gamma)])\leq \lambdap
  K_1\sum_{j=1}^{i-1} \lambdap^{j-1}$.  Then \eqref{eqn:homotopy2}
  implies $length([\tau(\gamma)]) \leq
  length([\sigma'\tau\sigma(\gamma)])+ K_1$, thus yielding the
  following which completes our induction:
  \begin{equation*}
    length([\tau(\gamma)])\leq \lambdap K_1\sum_{j=1}^{i-1} \lambdap^{j-1} 
    +K_1= K_1\sum_{j=1}^{i}\lambdap^{j-1}.
  \end{equation*}
  Therefore $length([\tau'\tau(\gamma)]) \leq
  K_3K_1\sum_{j=1}^{i}\lambdap^{j-1}$ where $K_3$ is the Lipschitz
  constant for $\tau'$. As before, using \eqref{eqn:homotopy1} we have
  $$length(\gamma) \leq K_3K_1\sum_{j=1}^{i}\lambdap^{j-1}+K_2.$$
  Setting $K=K_3K_1\frac{1}{\lambdap-1} + K_2$ completes
  the proof.
\end{proof}

Using Proposition \ref{prop:lb} we can estimate the difference between
the homotopic maps $\tau$ and $\sigma'^n\tau\sigma^n$.

\begin{lemma}\label{lm:Kn}
  Let $\sigma\co\Gamma \to \Gamma$ be as in Convention \ref{con:ng}.
  Then there is a a constant $K \geq 0$ such that for any path $\gamma
  \subset \Gamma$ and $n \geq 0$ we have
  \[ length([\tau(\gamma)]) \geq
  length([\sigma'^n\tau\sigma^n(\gamma)]) - K\mu^n. \]
\end{lemma}

\begin{proof}
  Since $\tau$ is homotopic to $\sigma'^n\tau\sigma^n$, if $\gamma$ is
  a closed loop then $length([\tau(\gamma)]) =
  length([\sigma'^n\tau\sigma^n(\gamma)])$.

  If $\gamma$ is not a loop, we can add a segment $\alpha$ of bounded
  length to $[\sigma^n(\gamma)]$ to get a closed loop $\gamma_1$ such
  that $[\sigma^n(\gamma)]$ is a subpath of $\gamma_1$ and $\alpha$ is
  an embedded arc.  Let $\gamma_0$ be a closed loop in $\Gamma$ such
  that $[\sigma^n(\gamma_0)] = \gamma_1$.  As $\alpha$ has bounded
  length, there is a constant $K_1$ such that
  $length([\sigma'^n\tau(\alpha)]) \leq K_1\mu^n$.  Hence:
  \begin{align*}
    length([\tau(\gamma_0)]) & = length([\sigma'^n\tau(\gamma_1)]) \\
    & \geq length([\sigma'^n\tau\sigma^n(\gamma)]) -
    length([\sigma'^n\tau(\alpha)]) \\
    & \geq length([\sigma'^n\tau\sigma^n(\gamma)]) - K_1\mu^n.
  \end{align*}
  There is path a $\beta$ in $\Gamma$ (unique up to homotopy) such
  that the concatenation of $\gamma$ and $\beta$ is homotopic to
  $\gamma_0$.  Notice that $[\sigma^n(\beta)] = \alpha$ as
  $[\sigma^n(\gamma)]$ is a subpath of $\gamma_1$. Now
  $length([\tau(\gamma)]) \geq length([\tau(\gamma_0)]) -
  length([\tau(\beta)])$.  We will show that $\beta$ is the union of a
  bounded number of $i$--vanishing paths where $i \leq n$ along with
  some segments of bounded length.

  Let $x$ and $y$ be the endpoints of $\alpha$ and $p$ and $q$ the
  endpoints of $\beta$ with $\sigma^n(p) = x$ and $\sigma^n(q) = y$.
  Now decompose $\beta = \beta_0 \cdot \beta_1$ by subdividing $\beta$
  at the point in $(\sigma^n)\inv(y)$ that is closest (along $\beta$)
  to $x$.  Then $\beta_1$ is an $i$--vanishing path for $\sigma$ with
  $i \leq n$ as $\alpha$ is embedded and $[\sigma^n(\beta)] = \alpha$.
  Thus $[\sigma^n(\beta_0)] = \alpha$.  Similarly, decompose $\beta_0
  = \beta_2 \cdot \beta_3$ where $\beta_2$ is an $i'$--vanishing path
  for $\sigma$ with $i' \leq n$ and where $[\sigma^n(\beta_2)] = x$.
  Thus $[\sigma^n(\beta_3)] = \alpha$ and $(\sigma^n)\inv(x)$ and
  $(\sigma^n)\inv(y)$ only intersect $\beta_3$ in its endpoints.  Now
  repeat at the vertices contained in $\alpha$ to decompose $\beta_3$
  as a union of vanishing paths (the number of which is bounded by the
  number of vertices of $\Gamma$) connected by segments that are
  homeomorphically mapped to the edges of $\alpha$.  The length and
  number of such segments is bounded.  Hence by Proposition
  \ref{prop:lb} and since $\tau$ induces a Lipschitz map between the
  universal covers of $\Gamma$ and $\Gamma'$ there is a constant $K_2$
  such that $length([\tau(\beta)]) \leq K_2\mu^n$.  Setting $K = K_1 +
  K_2$ completes the proof.
\end{proof}

To show the correct lower bound on the length of an $i$--vanishing
path for $\sigma\co \Gamma \to \Gamma$ we need the following lemma
from \cite{ar:BFH97}.

\begin{lemma}[Lemma 2.9 \cite{ar:BFH97}]\label{lm:BFH-2.9}
  Let $\sigma\co \Gamma \to \Gamma$ be a train-track map representing
  a fully irreducible outer automorphism $\phi$.  Then for every $C>0$
  there is a number $M>0$ such that if $\gamma$ is any path, then one
  of the following holds:
  \begin{itemize}
  \item[1.] $[\sigma^M(\gamma)]$ contains a legal segment of length
    $>C$.
    
  \item[2.] $[\sigma^M(\gamma)]$ has fewer illegal turns than
    $\gamma$.
    
  \item[3.] $\gamma$ is a concatenation $x \cdot y \cdot z$ with
    $[\sigma^M(y)]$ (periodic) Nielsen and $x$ and $z$ have length
    $\leq 2C$ and at most one illegal turn.
  \end{itemize}
\end{lemma}

Using the above we can show that when $\Gamma$ does not contain an
orbit of periodic Nielsen paths then short vanishing paths for $\sigma$
vanish quickly.  First we need to understand paths in $\Gamma$ that
could satisfy conclusion 3 in Lemma \ref{lm:BFH-2.9}.

\begin{lemma}\label{lm:2illegal}
  Let $\sigma\co\Gamma \to \Gamma$ be as in Convention \ref{con:ng}
  and suppose that $\Gamma$ does not contain an indivisible orbit of
  periodic Nielsen paths.  Then there exists an $N > 0$ such that if
  $\gamma$ is an $i$--vanishing path for $\sigma \co \Gamma \to
  \Gamma$ with at most 2 illegal turns, then $i \leq N$.
\end{lemma}

\begin{proof}
  Suppose otherwise, therefore we have sequence of $i_j$--vanishing
  paths $\gamma_j$ with at most 2 illegal turns 
  where $i_j < i_{j+1}$.  We will show that this implies that
  some power of $\sigma$ has a Nielsen path which contradicts our
  assumption. 

  First note that the lengths of $\gamma_j$ are uniformly bounded,
  since legal subpaths of vanishing paths have length bounded by the
  critical constant.  By passing to a subsequence we can assume that
  the paths $\gamma_j$ all have the same combinatorial type, i.e, they
  cross the same turns of $\Gamma$ in the same order.  Further, by
  passing to a power of $\sigma$, possibly replacing the sequence
  $\gamma_j$ by $\sigma^\ell(\gamma_j)$ for some $\ell$ such that the
  vertex (or vertices) at the illegal turn is (are) fixed by $\sigma$
  and again passing to a subsequence, we can assume that
  $[\sigma(\gamma_j)] \subseteq \gamma_j$.

  We will break the proof up into two cases depending on whether the
  paths $\gamma_j$ contain a single illegal turn or two illegal turns.
  First assume that the paths $\gamma_j$ only contain a single illegal
  turn.  Notice that any $i$--vanishing path with a single illegal
  turn has the form $a \cdot b$ where $a$ and $b$ are legal segments
  of the same length such that $\sigma^i(a) = \sigma^i(\bar{b})$.
  Further any subpath of the form $a' \cdot b'$ where $a'$ and $b'$
  are legal segments of the same length is an $i'$--vanishing path
  where $i' \leq i$.

  We have two claims about the vanishing paths $\gamma_j$.

  \medskip \noindent \textbf{Claim 1}: If $\iota$ is a subpath of
  $\gamma_j$ and $\iota$ is an $i$--vanishing path for $\sigma\co
  \Gamma \to \Gamma$ then $i \leq i_j$.

  \begin{proof}[Proof of Claim]
    If $[\sigma^{i_j}(\iota)]$ is not a point, then it is a path that
    does not contain any illegal turns (as there is a single illegal
    turn in $\gamma_j$) and hence $\iota$ is not a vanishing path.
  \end{proof}

  \medskip \noindent \textbf{Claim 2}: If $j < k$ then $\gamma_j$ is a
  subpath of $\gamma_k$.

  \begin{proof}[Proof of Claim]
    Since the two legal segments in each of $\gamma_j$ and $\gamma_k$
    have equal lengths, we must have $\gamma_j \subset \gamma_k$ or
    $\gamma_k \subset \gamma_j$ because they have the same
    combinatorial type. The latter is not possible by Claim 1.
  \end{proof}
  
  Therefore by the above Claims, we have $\gamma_j \subset
  \gamma_{j+1}$ and hence there is a well-defined limit
  $\gamma_\infty$.  By construction, we have $\gamma_\infty = a \cdot
  b$ where $a$ and $b$ are legal segments of the same length as this
  holds for each of the $\gamma_j$.  Further $[\sigma(a)]$ and
  $[\sigma(b)]$ are also legal segments of the same length and
  $[\sigma(\gamma_\infty)] \subseteq [\gamma_\infty]$.  If
  $[\sigma(\gamma_\infty)] \neq \gamma_\infty$ then
  $[\sigma(\gamma_\infty)] \subset \gamma_{j'}$ for some $j'$ and as
  $[\sigma(a)]$ and $[\sigma(b)]$ are also legal segments of the same
  length, $[\sigma(\gamma_\infty)]$ is a $i'$--vanishing path for some
  $i' \leq i$.  However, since $\gamma_j \subset \gamma_\infty$ are
  all $i_j$--vanishing paths with $i_j \to \infty$, by Claim 1, the
  path $\gamma_\infty$ is not a vanishing path.  Therefore
  $[\sigma(\gamma_\infty)] = \gamma_\infty$ and hence $\gamma_\infty$
  is a Nielsen path.

  It remains to consider the case when $\gamma_j$'s have two illegal
  turns. We will argue that this is not possible. Suppose $\gamma$ is
  a vanishing path with two illegal turns and
  $[\sigma(\gamma)]\subset\gamma$. The middle legal segment $b$ of
  $\gamma=a\cdot b\cdot c$ maps over itself and therefore has a unique
  fixed point, breaking it up as $b=b_1b_2$. There are two subcases
  depending on whether $b$ maps over itself preserving or reversing
  the orientation; the cases are similar and we assume the orientation
  is preserved. Therefore the iterates of $a \cdot b_1$ and of
  $b_2\cdot c$ never cancel against each other, so both must be
  vanishing. But $[\sigma(a\cdot b_1)]$ has the form $a'\cdot b_1$ for
  some $a'\subset a$ (and both $ab_1$ and $a'b_1$ are vanishing). This
  contradicts the fact that vanishing paths with one illegal turn have
  legal segments of equal lengths.

  This completes the proof of the lemma.
\end{proof}

\begin{lemma}\label{lm:ng-vp}
  Let $\sigma\co\Gamma \to \Gamma$ be as in Convention \ref{con:ng}
  and suppose that $\Gamma$ does not contain an indivisible orbit of
  periodic Nielsen paths.  For every $L > 0$ there exists an $N > 0$,
  such that if $\gamma$ is an $i$--vanishing path for $\sigma$ and
  $length(\gamma) \leq L$ then $i \leq N$.
\end{lemma}

\begin{proof} 
  As there is a lower bound on the distance between any two illegal
  turns in $\Gamma$, for any vanishing path with length less than $L$,
  the number of illegal turns in a path of length at most $L$ is at
  most some constant $I$.

  Let $C$ be larger than the critical constant for $\sigma$ and $M$
  the constant from Lemma~\ref{lm:BFH-2.9}.  If $[\sigma^M(\gamma)]$
  had a legal segment of length greater than $C$, this would imply
  that $[\sigma^n(\gamma)]$ had positive length for all $n > M$, and
  thus $\gamma$ is not a vanishing path, contrary to hypothesis. Thus
  possibility 1 in Lemma \ref{lm:BFH-2.9} cannot occur for any
  vanishing path. On the other hand, if possibility 3 of
  Lemma \ref{lm:BFH-2.9} occurs then $y$ is trivial since by
  assumption $\Gamma$ does not contain an indivisible orbit of
  periodic Nielsen paths, hence $\gamma = x \cdot z$ and has length
  less than $2C$ and at most 2 illegal turns.

  Thus the possibility 2 of Lemma~\ref{lm:BFH-2.9} must occur for
  every vanishing path with more than 2 illegal turns or length
  greater than $2C$, namely $[\sigma^M(\gamma)]$ has strictly fewer
  illegal turns than $\gamma$.  Hence $[\sigma^{IM}(\gamma)]$ has at
  most two illegal turns and length less than $2C$ for any vanishing
  path $\gamma$.  Therefore, we see that $[\sigma^{N_0+IM}(\gamma)]$
  is a point, where $N$ is the constant from Lemma \ref{lm:2illegal}
  and hence $i \leq N + IM$.
\end{proof}

To derive the more precise lower bound on the length of an
$i$--vanishing path, we consider two types of \emph{legality}.  For a
path $\gamma$ with $I$ illegal turns, we define the ratio:
\[LEG_1(\gamma) = \frac{length(\gamma)-2\lambda\inv
  BI}{length(\gamma)}\] where $B = BCC(\sigma)$ is the bounded
cancellation constant.  Since $length([\sigma(\gamma)]) \geq \lambda
length(\gamma) - 2BI$ we see that $length([\sigma(\gamma)]) \geq
\lambda LEG_1(\gamma) length(\gamma)$.  The following Lemma and
Corollary do not depend on the absence of an orbit of periodic Nielsen
paths.

\begin{lemma}\label{lm:legality-converge}
  Let $\sigma\co\Gamma \to \Gamma$ be as in Convention \ref{con:ng}.
  For any path $\gamma \subset \Gamma$ with $LEG_1(\gamma) >
  \lambda\inv$, the ratio $LEG_1([\sigma^n(\gamma)])$ converges to
  $1$.  Moreover, for any $0 < \epsilon < 1 - \lambda^{-1}$, there is
  a $\delta > 0$ such that if $ LEG_1(\gamma) \geq \lambda\inv +
  \epsilon$ then the infinite product $\prod_{n=0}^\infty
  LEG_1([\sigma^n(\gamma)])$ converges to a positive number greater
  than or equal to $\delta$.
\end{lemma}

\begin{proof}
  Since the number of illegal turns in $[\sigma(\gamma)]$ is at most
  the number of illegal turns in $\gamma$ and
  $length([\sigma(\gamma)]) \geq \lambda length(\gamma) - 2BI$ we see:
  \begin{align*}
    \frac{1 - LEG_1([\sigma(\gamma)])}{1 - LEG_1(\gamma)} &\leq
    \frac{2\lambda\inv BI}{\lambda length(\gamma) - 2
      BI}\left(\frac{2\lambda\inv BI}
      {length(\gamma)}\right)\inv \\
    &= \frac{length(\gamma)}{\lambda length(\gamma) - 2BI} =
    \frac{1}{\lambda LEG_1(\gamma)} < 1.
  \end{align*}
  Hence $LEG_1([\sigma(\gamma)]) > LEG_1(\gamma) > \lambda\inv$ and so
  repeating the above we see that $LEG_1([\sigma^n(\gamma)])$
  converges to $1$.  By bounding $LEG_1(\gamma)$ away from
  $\lambda\inv$ we can make the rate of convergence independent of the
  path $\gamma$.  Further the above calculation shows:
  $$\limsup_{n \to \infty} \frac{1 - 
    LEG_1([\sigma^{n+1}(\gamma)])}{1-LEG_1([\sigma^n(\gamma)])} \leq
  \lambda\inv.$$ Hence the product $\prod_{n=1}^\infty
  LEG_1([\sigma^n(\gamma)])$ converges to a positive number, which can
  be bounded away from 0 independent of the path $\gamma$ by bounding
  $LEG_1(\gamma)$ away from $\lambda\inv$.
\end{proof}

The following corollary follows immediately.

\begin{corollary}\label{co:legality-lb}
  Let $\sigma\co \Gamma \to \Gamma$ be as in Convention \ref{con:ng}.
  For any $0 < \epsilon < 1 - \lambda^{-1}$, there is a constant $K >
  0$ such that for any $n \geq 0$ and any path $\gamma \subset \Gamma$
  with $LEG_1(\gamma) \geq \lambda\inv +
  \epsilon$: $$length([\sigma^n(\gamma)]) \geq K\lambda^n
  length(\gamma).$$
\end{corollary}

As a word of caution, we will be applying this corollary to paths in
$\Gamma'$ and the train-track map $\sigma'\co \Gamma' \to \Gamma'$.
Specifically, this corollary enables us to get a lower bound on the
length of a path $\gamma \subset \Gamma$ where $[\tau(\gamma)] \subset
\Gamma'$ has a large legality with respect to $\sigma'\co\Gamma' \to
\Gamma'$.

\begin{lemma}\label{lawabiding implies growth}   
  Let $\sigma\co\Gamma \to \Gamma$ be as in Convention \ref{con:ng}.
  For any $0 < \epsilon < 1 - \lambdap^{-1}$, there exist constants $K
  > 0$ and $C,C' \geq 0$ such that for any $n \geq 0$ and any path
  $\gamma \subset \Gamma$, if $LEG_1([\tau\sigma^n(\gamma)]) \geq
  \lambdap^{-1} + \epsilon$ then: $$length(\gamma) \geq K \lambdap^n
  (length([\tau\sigma^n(\gamma)]) - C) - C'.$$
\end{lemma}

\begin{proof}    
  Applying Corollary \ref{co:legality-lb} to $\sigma'\co \Gamma' \to
  \Gamma'$ we see that there is a constant $K_1 > 0$ such that for any
  path $\gamma \subset \Gamma$ with $LEG_1([\tau\sigma^n(\gamma)])
  \geq \lambda^{-1} + \epsilon$:
  \[length([\sigma'^n\tau\sigma^n(\gamma)]) \geq K_1\lambdap^n
  length([\tau\sigma^n(\gamma)]).\] Combining this with Lemma
  \ref{lm:Kn} there is a constant $C \geq 0$ such that:
  \[ length([\tau(\gamma)]) \geq
  K_1\lambdap^n(length([\tau\sigma^n(\gamma)]) - C). \] Finally, since
  $\tau$ induces a quasi-isometry between the universal covers of
  $\Gamma$ and $\Gamma'$, there are constants $K > 0$ and $,C' \geq 0$
  such that \[length(\gamma) \geq
  K\lambdap^n(length([\tau\sigma^n(\gamma)]) - C) - C'\].
\end{proof}

We now show that if $\gamma$ is an $i$--vanishing path for $\sigma$
then $LEG_1(\tau(\gamma))$ can be made close to 1 for large enough
$i$, thus enabling us to use Lemma \ref{lawabiding implies growth}.
To show this we need to use the version of legality from
\cite{ar:BFH97}.  Let $C$ be larger than critical constant for
$\sigma$ and the critical constant for $\sigma'$ and
define: $$LEG_2(\gamma) = \frac{\mbox{\textit{sum of the lengths of
      the legal segments of $\gamma$ of length $\geq
      C$}}}{length(\gamma)}$$There is a constant $\eta$ such that,
$length([\sigma^n(\gamma)]) \geq
\eta\lambda^nLEG_2(\gamma)length(\gamma)$ for any path $\gamma$.  We
need the following lemma regarding this version of legality.

\begin{lemma}[Lemma 5.6 \cite{ar:BFH97}]\label{lm:BFH-5.6}
  Assume that $\phi$ has no nontrivial periodic conjugacy classes (or
  equivalently that $\phi$ is not represented by a homeomorphism of a
  punctured surface).  Then there are constants $L,N,\epsilon > 0$
  such that for any path $\gamma$ with $length(\gamma) \geq L$ and
  every $n \geq N$ either $LEG_2(\sigma^n(\gamma)) \geq \epsilon$ or
  $LEG_2(\sigma'^n\tau(\gamma)) \geq \epsilon$.
\end{lemma}

In \cite{ar:BFH97}, the above lemma is stated for conjugacy classes in
$F_k$ but the proof applies equally well to this setting.

\begin{lemma}\label{lm:vp-legality}
  Let $\sigma\co \Gamma \to \Gamma$ be as in Convention \ref{con:ng}
  and suppose that $\Gamma$ does not contain an indivisible orbit of
  periodic Nielsen paths.  For all $0 < \delta < 1$, there are
  constants $L$ and $N$ such that if $\gamma$ is a vanishing path for
  $\sigma$ with length $\geq L$ then $LEG_1([\sigma'^n\tau(\gamma)])
  \geq \delta$ for all $n \geq N$.
\end{lemma}

\begin{proof}
  Let $L_1$, $N_1$ and $\epsilon$ be constants from Lemma
  \ref{lm:BFH-5.6}.  Since $\gamma$ is a vanishing path for $\sigma$,
  we have $LEG_2(\sigma^n(\gamma)) = 0$ for all $n \geq 0$ and hence
  we must have the second conclusion from this Lemma, namely:
  $LEG_2([\sigma'^{N_1}\tau(\gamma)]) \geq \epsilon$ for vanishing
  paths $\gamma$ with length at least $L$.

  Therefore there is a constant $\eta>0$ such that for any vanishing
  path with length at least $L$: $$length([\sigma'^n\tau(\gamma)])
  \geq \eta\epsilon
  \lambdap^{n-N_1}length([\sigma'^{N_1}\tau(\gamma)])$$ for $n \geq
  N_1$.  Further, there is a constant $I$ such that the number of
  illegal turns in $\sigma'^n\tau(\gamma)$ is at most
  $Ilength([\sigma'^{N_1}\tau(\gamma)])$ for all $n \geq N_1$ (since
  legal turns go to legal turns and the distance between illegal turns
  is uniformly bounded from below).  Hence set $L = L_1$ and $N$ large
  enough such that $1 -
  \frac{2B\lambdap^{-1}I}{\eta\epsilon\lambdap^{N-N_1}} \geq \delta$.
\end{proof}

Finally, we can prove that the $i$--vanishing path have length
approximately $\lambdap^i$ when there are no orbits of periodic
Nielsen paths.

\begin{proposition}\label{prop:ngub}
  Let $\sigma\co \Gamma \to \Gamma$ be as in Convention \ref{con:ng}
  and suppose that $\Gamma$ does not contain an indivisible orbit of
  periodic Nielsen paths.  There exists a constants $K > 0$ and $C
  \geq 0$ such that for every $i$--vanishing path $\gamma$ for
  $\sigma\co\Gamma \to \Gamma$ we have $length(\gamma) \geq
  K\lambdap^i - C$.
 \end{proposition}

\begin{proof}
  Fix a small $0 < \epsilon < 1 - \lambdap^{-1}$, and let $L_1$ and
  $N_1$ be the constants from Lemma \ref{lm:vp-legality} using $\delta
  = \lambdap^{-1} + \epsilon$.  Also let $K_1,C_1,C'_1$ be the
  constants from Lemma \ref{lawabiding implies growth} using
  $\epsilon/2$.  Let $\ell$ be large enough such that the
  $\tau$--image of an $\ell$--vanishing path has length at least
  $2C_1$ and such that an $(\ell - N_1)$--vanishing path has length at
  least $L$.  Such an $\ell$ exists by Lemma \ref{lm:ng-vp}.

  Suppose $\alpha$ is an $\ell$-vanishing path for $\sigma$.  Hence
  $\sigma^{N_1}(\alpha)$ has length at least $L_1$ and by Lemma
  \ref{lm:vp-legality}:
  $$LEG_1([\sigma'^{N_1}\tau\sigma^{N_1}(\alpha)]) \geq \lambdap^{-1} +
  \epsilon.$$ Since $\tau$ is homotopic to
  $\sigma'^{N_1}\tau\sigma^{N_1}$, by further requiring that $\ell$ be
  sufficiently large, we can guarantee that the length of an
  $\ell$--vanishing path is large enough such
  that: $$LEG_1([\tau(\alpha)]) \geq
  LEG_1([\sigma'^{N_1}\tau\sigma^{N_1}(\alpha)]) - \epsilon/2 \geq
  \lambdap^{-1} + \epsilon/2.$$

  For $i > \ell$ we can apply the above to $\alpha =
  \sigma^{i-\ell}(\gamma)$, by Lemma \ref{lawabiding implies growth}
  we have: $$length(\gamma) \geq
  K_1\lambdap^{i-\ell}(length([\tau\sigma^{i-\ell}(\gamma)]) - C_1) -
  C'_1.$$ Since $length([\tau\sigma^{i-\ell}(\gamma)]) \geq 2C_1$,
  have $length(\gamma) \geq K_1\lambdap^{i-\ell}C_1 - C'_1$.
\end{proof}


Recall that $T^n_e$ is the span of the points in $(f^n)\inv(\uc{p_e})$
where $f\co T \to T$ is a lift of $\sigma \co \Gamma \to \Gamma$ and
$p_e$ is a lift of a point in the interior of $e$.  The path joining
any pair of points in $(f^{n})\inv(p_{e})$ is an $i$--vanishing path
for some $i \leq n$.  We define a set of equivalence relations on the
set $(f^{n})\inv(p_{e})$ by $x \sim_{i} x'$ if the path connecting
them is $j$--vanishing path where $j \leq i$.  Equivalently, $x \sim_i
x'$ if $f^i(x) = f^i(x')$ for $i \leq n$.  Notice that if $x \sim_j
x'$ and $x' \sim_i x''$ for $j \leq i$ then $x \sim_i x''$.  An
$\sim_{i}$ equivalence class is called an \textit{$i$--clump}.  Notice
that the entire set $(f^n)\inv(p_e)$ is an $n$--clump.  We will also
sometimes use the term $i$--clump to refer to the subtree of $T^n_e$
spanned by the points in an individual $i$--clump.

Contained within an $i$--clump there are several $j$--clumps where
$j<i$.  A \textit{$j$--spanning tree contained in an $i$--clump} is a
subtree spanned by choosing a point in each $j$--clump contained with
$i$--clump.  We will estimate the volume of $T^n_e$ by estimating the
size and number of $i$--spanning trees contained in $T^n_e$.

\begin{lemma}\label{lm:numberjclumps}
  Let $\sigma\co \Gamma \to \Gamma$ be as in Convention \ref{con:ng}
  and $f\co T \to T$ a lift of $\sigma$.  With the notation above, for
  any $j < i \leq n$, the number of $j$--clumps contained in an
  $i$--clump of $T^n_e$ is $\sim \lambda^{i-j}$.
\end{lemma}

\begin{proof}
  As $\sigma\co \Gamma \to \Gamma$ is irreducible and linearly expands
  any edge by the factor $\lambda$, there exist constants $K_1 \geq 1$
  and $C_1 \geq 0$ such that for any point $x \in \Gamma$, we have
  $\frac{1}{K_1}\lambda^\ell - C_1 \leq$ the number of points in
  $\sigma^{-\ell}(x) \leq K_1\lambda^\ell + C_1$.

  Now after applying $f^{j}$ to $T$, any $j$--clump becomes a point
  and distinct $j$--clumps become distinct points.  Further applying
  $f^{i-j}$, the image of the $j$--clumps which are contained within
  an $i$--clump all map to the same point $x$.  Hence the $j$--clumps
  contained within an $i$--clump are parametrized by
  $\sigma^{-(i-j)}(p_e)$.
\end{proof}
  
\begin{lemma}\label{lm:spanning}
  Let $\sigma\co \Gamma \to \Gamma$ be as in Convention \ref{con:ng}
  and $f\co T \to T$ a lift of $\sigma$ and suppose that $\Gamma$ does
  not contain an indivisible orbit of periodic Nielsen paths.  Fix an
  $\ell > 0$, then for any $\ell < i \leq n$, the volume of a
  $(i-\ell)$--spanning tree contained in an $i$--clump is $\sim
  \lambdap^i$.
\end{lemma}

\begin{proof}
  As in Lemma \ref{lm:numberjclumps}, the number of $(i -
  \ell)$--clumps contained in an $i$--clump is bounded above and below
  by some constants independent of $i$.  Thus any $(i-\ell)$--spanning
  tree is covered by a bounded number of $i$--vanishing paths.  By
  Propositions \ref{prop:lb} and \ref{prop:ngub} the length of an
  $i$--vanishing path is bounded above and below by linear functions
  of $\lambdap^i$.
\end{proof}    

Putting these together we can show:

\begin{proposition}\label{prop:ngvolslice}
  Let $\sigma\co \Gamma \to \Gamma$ be as in Convention \ref{con:ng}
  and $f\co T \to T$ a lift of $\sigma$ and suppose that $\Gamma$ does
  not contain an indivisible orbit of periodic Nielsen paths.  Then
  for any edge $e \subset T$:
  \[ \vol(T_e^n) \sim \lambda^n + \lambda^{n-1}\lambdap + \cdots +
  \lambda\lambdap^{n-1} + \lambdap^n.\]
\end{proposition}

\begin{proof}
  As $T^n_e$ is covered by $i$--vanishing paths for $i \leq n$, we can
  also cover $T^n_e$ using $(i-1)$--spanning trees contained in
  $i$--clumps.  The number of such $(i-1)$--spanning trees is
  equal to the number of $i$--clumps.  Thus putting together Lemmas
  \ref{lm:numberjclumps} and \ref{lm:spanning} we get that
  $\vol(T^n_e)$ is bounded above by a linear function of $\lambda^n +
  \lambda^{n-1}\lambdap + \cdots + \lambda\lambdap^{n-1} +
  \lambdap^n$.

  To get the lower bound we need to look carefully at the overlap of
  the these spanning trees.  Let $K$ by the constant appearing in
  Proposition \ref{prop:lb} and choose $\ell$ such that
  $\lambdap^\ell$ is much bigger than $K$.  We will proceed by
  induction on $i$ and show that the volume of an $i$--clump is
  bounded below by a linear function of $\lambdap^i +
  \lambda^\ell\lambdap^{i - \ell} + \lambda^{2\ell}\lambdap^{i- 2\ell}
  + \cdots$.

  Fix an $i$--clump and let $T_0$ be an $(i - \ell)$--spanning tree
  contained within this $i$--clump.  Let $j = i - \ell$ and denote the
  $j$--clumps contained within this $i$--clump by
  $T_1,\ldots,T_{M_j}$.  By induction, the volume of any of the $T_m$
  is bounded below by a linear function of $\lambdap^j +
  \lambda^\ell\lambdap^{j - \ell} + \lambda^{2\ell}\lambdap^{j- 2\ell}
  + \cdots$.  Notice that $M_j \approx \lambda^\ell$.  For any $T_m$,
  the overlap of $T_0$ and $T_m$ has volume bounded above by
  $K\lambdap^j$.  This follows as the overlap is contained within a
  $j$--vanishing path.  Therefore the volume of $\widehat{T}_0$, the
  subforest of $T_0$ obtained by removing any overlap with some $T_m$
  has volume bounded below by a linear function of $\lambdap^i$.
  Applying induction this shows that the volume of an $i$--clump is
  bounded below by a linear function of $\lambdap^i +
  \lambda^\ell\lambdap^{i - \ell} + \lambda^{2\ell}\lambdap^{i- 2\ell}
  + \cdots \sim \lambdap^i + \lambda\lambdap^{i-1} + \cdots +
  \lambda^{i-1}\lambdap + \lambda^i$.  For $i=n$, we achieve the
  desired lower bound for $\vol(T^n_e)$.
\end{proof}

Putting together Propositions \ref{prop:estimate} and
\ref{prop:ngvolslice} we get the second case of Theorem
\ref{th:growth} from the Introduction.

\begin{theorem}\label{th:nongeo}
  Suppose $\phi \in \Out(F_k)$ is a fully irreducible automorphism
  with expansion factor $\lambda$ and a nongeometric stable tree.  Let
  $\lambdap$ be the expansion factor of $\phi\inv$.  Then for any
  $T,T' \in cv_k$:
  \[i(T,T'\phi^n) \sim \lambda^n + \lambda^{n-1}\lambdap + \cdots +
  \lambda\lambdap^{n-1} + \lambdap^n.\]
\end{theorem}


\subsection{\texorpdfstring{$T^+$ geometric}{T+ geometric}}\label{ssc:geo}

We will now look at the case when $T^+$ is geometric.  We need the
following well-known result providing a special representative of an
automorphism with geometric stable tree.

\begin{proposition}\label{prop:parageo-map}
  Suppose that $T^+$, the stable tree of a fully irreducible
  automorphism $\phi$, is geometric.  Then some power of $\phi$ has a
  train-track representative $\sigma\co \Gamma \to \Gamma$ such that:
\begin{itemize}
    \item[1.]  $\sigma$ has exactly one indivisible Nielsen path;
    
    \item[2.]  $\sigma$ has a unique illegal turn;
     
    \item[3.]  $\sigma$ is homotopic to the result of iteratively
      folding the indivisible Nielsen path along the illegal turn.

\end{itemize}
\end{proposition}

\begin{proof}
  Let $\sigma\co \Gamma \to \Gamma$ be a stable train-track
  representative for $\phi$. Then $\sigma$ contains a unique
  indivisible orbit of periodic Nielsen paths, each path of which
  contains a unique illegal turn.  This follows by Lemmas 3.4 and 3.9
  of \cite{ar:BH92}. Thus replacing $\sigma$ by a power, we can
  achieve 1.  Further, by Lemma 3.9 of \cite{ar:BH92}, after replacing
  $\sigma$ by this power, $\sigma\co \Gamma \to \Gamma$ has a unique
  illegal turn, hence 2 holds.

  Any homotopy equivalence between graphs factors into a sequence of
  folds \cite{ar:S83}.  Since edges in a legal turn do not get
  identified by $\sigma$, the first fold in the factored sequence is
  folding the indivisible Nielsen path.  Folding a pair of edges of a
  stable train-track map results in a stable train-track map (remark
  on page 23 in \cite{ar:BH92}).  Hence the factored sequence is
  iteratively folding the indivisible Nielsen path.
\end{proof}

We train-track representative satisfying the conclusions of the
preceding proposition is called a \emph{parageometric} train-track
representative.  For $t \geq 0$ let $\Gamma_t$ be the graph resulting
from $\Gamma$ by iteratively folding the indivisible Nielsen path along
the illegal turn at a constant rate, where for $t \in \Z_{\geq 0}$ the
induced map $\sigma_{0t}\co \Gamma \to \Gamma_t$ is $\sigma^t$.  For
$0\leq s\leq t$ let $\sigma_{st}\co \Gamma_s \to \Gamma_t$ by the
induced map.

The dichotomy in Theorem \ref{th:growth} between geometric and
nongeometric growth rates stems from the following observation.
Suppose $\sigma\co \Gamma \to \Gamma$ has an indivisible Nielsen path
$\gamma$.  For small $\epsilon > 0$ let $\gamma_\epsilon$ be the
subpath of $\gamma$ obtained by removing $\epsilon$ neighborhoods of
its endpoints.  Then $\gamma_\epsilon$ is an $i$--vanishing path for
$\sigma$ and moreover, $i \to \infty$ as $\epsilon \to 0$.  This
observation follows as we can write $\gamma$ as a concatenation of
legal paths $\gamma = a_0 \cdot b_0\cdot
\overline{b_1}\cdot\overline{b_1}$ where $\sigma(a_i) = a_i\cdot b_i$
and $\sigma(b_0) = \sigma(b_1)$.  Hence for large $n$ (depending on
$\epsilon$), $[\sigma^n(\gamma_\epsilon)] \subset b_0 \cdot
\overline{b_1}$ and therefore is an $n+1$--vanishing path.  Therefore,
in contrast with Lemma \ref{lm:ng-vp} and Proposition \ref{prop:ngub}
in the nongeometric case, there are bounded length $i$--vanishing
paths for large $i$.  In fact, in the geometric case we claim that any
$i$--vanishing path is a composition tightened relative to its
endpoints of vanishing paths that are contained in this Nielsen path.
A vanishing path contained in an indivisible Nielsen path is called a
\emph{special vanishing path}.

We begin with a simple criterion for finding special vanishing paths.

\begin{lemma}\label{lm:special}
  Suppose $\sigma \co \Gamma \to \Gamma$ is a parageometric
  train-track representative.  There is a constant $T > 0$ such that
  any vanishing path for $\sigma_{0t}\co \Gamma \to \Gamma_t$ for $t
  \leq T$ is a special vanishing path.
\end{lemma}

\begin{proof}
  Choose $T$ such that $\sigma_{0t}\co \Gamma \to \Gamma_{t}$ does not
  identify a pair of vertices of $\Gamma$ for $t \leq T$.  Then any
  vanishing path of $\sigma_{0t}$ is contained within a pair of edges
  that are partially folded together by $\sigma_{0t}$, and the turn
  between these edges is illegal.  As $\sigma$ has a unique illegal
  turn and this illegal turn is contained in the unique Nielsen path,
  any vanishing path for $\sigma_{0t}$ is contained in the Nielsen
  path and hence is a special vanishing path.
\end{proof}

We can now show that vanishing paths can be covered by special
vanishing paths.

\begin{proposition}\label{prop:geo-vp}
  Suppose $\sigma \co \Gamma \to \Gamma$ is a parageometric
  train-track representative and $\gamma$ is $i$--vanishing path for
  $\sigma$.  Then $\gamma = [\gamma_1 \cdots \gamma_\ell ]$ where for
  $j = 1,\ldots,\ell$, the path $\gamma_j$ is a special vanishing
  path.
\end{proposition}

\begin{proof}
  Since folding $\Gamma$ does not collapse loops and since $\Gamma$
  has a unique illegal turn, for any path $\alpha$ in $\Gamma$, there
  is a lower bound on the distance between two illegal turns in
  $\sigma_{0t}(\alpha)$ independent of $\alpha$ and $t$.  Therefore,
  there is an $\epsilon$ such that if $\alpha$ is a vanishing path
  with $length(\alpha) < \epsilon$ then $\alpha$ contains a single
  illegal turn.  Hence, there is a $\delta > 0$ such that if $\gamma$
  is a vanishing path for $\sigma_t\co \Gamma \to \Gamma$ with $t$
  minimal, then $[\sigma_{0t-\delta}(\gamma)]$ contains a single
  illegal turn.

  We will prove the proposition by induction.  The inductive claim is
  as follows: if $\gamma$ is a vanishing path for
  $\sigma_{0t_\gamma}\co \Gamma \to \Gamma_{t_\gamma}$ with $t_\gamma$
  minimal, then $\gamma = [a_1 \cdot a_0 \cdot a_2 ]$ where $a_0$ is a
  special vanishing path and $a_1$ and $a_2$ are vanishing paths for
  $\sigma_{0t_1}\co \Gamma \to \Gamma_{t_1}$ where $t_1 = t_\gamma -
  \delta$.  The basecase, where $t_\gamma \leq T$ where $T$ is the
  constant from Lemma \ref{lm:special} is proved by Lemma
  \ref{lm:special} as in this case $\gamma = a_0$ is a special
  vanishing path.
  
  By construction a single illegal turn in $\gamma$ is still present
  in $[\sigma_{0,t_1}(\gamma)]$.  In the universal cover of $\Gamma$
  there are lifts of the vanishing path $\gamma$ and the Nielsen path
  that share a lift of this illegal turn.  Write the lift of $\gamma$
  as $\gamma_1\cdot \gamma_0\cdot \gamma_2$ where $\gamma_0$ is the
  common overlap between the lift of $\gamma$ and the lift of the
  Nielsen path.  This decomposes the lift of the Nielsen path into
  $\beta_1 \cdot \gamma_0 \cdot \beta_2$.  Notice that $\beta_1$ and
  $\beta_2$ are legal paths.

  Since $[\sigma_{0t_1}(\gamma)]$ only contains a single illegal turn,
  as in Lemma \ref{lm:special} we see that $[\sigma_{0t_1}(\gamma)]$
  is contained in the image of the Nielsen path in $\Gamma_{t_1}$.
  Therefore, we can find a subpath $\beta'_1$ of $\beta_1$ such that
  $\gamma_1\cdot \overline{\beta'_1}$ is a vanishing path for
  $\sigma_{0t_1}$.  Similarly we can find a subpath $\beta'_2$ of
  $\beta_2$ such that $\gamma_2\cdot \overline{\beta'_2}$ is also a
  vanishing path for $\sigma_{0t_1}$.    Then
  $\beta'_1 \cdot \gamma_0 \cdot \beta'_2$ is a special vanishing
  path.  See Figure \ref{fig:svp}.  
  \begin{figure}[h]
    \psfrag{g1}{$\gamma_1$}
    \psfrag{g0}{$\gamma_0$}
    \psfrag{g2}{$\gamma_2$}
    \psfrag{b1}{$\beta'_1$}
    \psfrag{b2}{$\beta'_2$}
    \centering
    \includegraphics{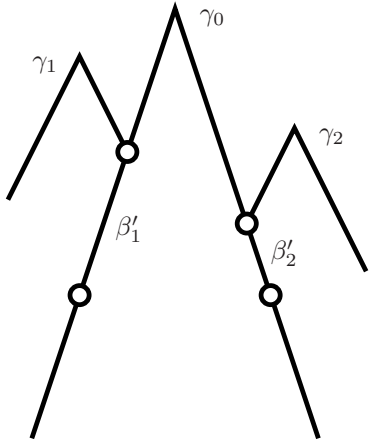}
    \caption{Decomposing the vanishing path $\gamma$ in the proof of
      Proposition \ref{prop:geo-vp}.}
    \label{fig:svp}
  \end{figure}
  By induction, we know that both $\gamma_1 \cdot\overline{\beta'_1}$
  and $\gamma_2\cdot \overline{\beta'_2}$ are the compositions of
  special vanishing paths pulled tight.  Therefore, $\gamma$ is the
  composition of special vanishing paths pulled tight.
\end{proof}

We can now prove the first case of Theorem \ref{th:growth} from the
Introduction.

\begin{theorem}\label{th:geo}
  Suppose $\phi \in \Out(F_k)$ is a fully irreducible automorphism
  with expansion factor $\lambda$ and a geometric stable tree.  Then
  for any $T',T'' \in cv_k$: \[i(T',T''\phi^n) \sim \lambda^{n}.\]
\end{theorem}

\begin{proof}
  First suppose that $\phi$ is represented by a parageometric
  train-track map $\sigma \co \Gamma \to \Gamma$.  This lifts to a map
  $f\co T \to T$.  By Proposition \ref{prop:estimate} $i(T',T''\phi^n)
  \sim \vol(T^n_e)$ where $T^n_e \subset T$ is the subtree spanned by
  the points in $(f^n)\inv(p_e)$ for any edge $e \subset T$.  Then
  $T^n_e$ as is a union of $i$--vanishing paths for $i \leq n$ and by
  Proposition \ref{prop:geo-vp} it is also covered special vanishing
  paths.  As in Lemma \ref{lm:numberjclumps} the number of special
  vanishing paths needed to cover $T^n_e$ is $\sim \lambda^n$.

  If $\phi$ is not represented by a parageometric train-track map,
  then by Proposition \ref{prop:parageo-map} some power $\phi^\ell$
  is.  Then for any $i$ we have $i(T',T''\phi^{i + n\ell}) =
  i(T',(T''\phi^{i})\phi^{n\ell}) \sim (\lambda^{\ell})^n \sim
  \lambda^{i + n\ell}$, hence $i(T',T''\phi^n) \sim \lambda^n$.
\end{proof}

Combining Theorems \ref{th:nongeo} and \ref{th:geo} we get Theorem
\ref{th:growth} from the Introduction.  We conclude with an example of
a parageometric automorphism, illustrating the difference between the
length of vanishing paths in the geometric direction and nongeometric
direction.

\begin{example}\label{ex:geom-nongeom}
  In this example we present subtrees $T^n_a$ for the following fully
  irreducible automorphisms:
  \begin{equation*}
  \begin{array}{rclcrcl} 
    a & \mapsto & ac & \quad & a & \mapsto & b \\
    \phi\co b & \mapsto & a & \quad & \psi\co b & \mapsto & c \\
    c & \mapsto & b & \quad & c & \mapsto & ab
  \end{array}
  \end{equation*}
  As in Example \ref{ex:algorithm} we let $T$ be the universal cover
  of $R_3$, the $3$--rose marked with petals labeled $a,b,c$ and let
  $f_\phi\co T \to T$ and $f_\psi \co T \to T$ denote the lift of the
  obvious homotopy equivalences of $R_3$ representing $\phi$ and
  $\psi$ respectively.  Figure \ref{fig:parageom+} shows
  $(f_\phi^n)\inv(p_a)$ and Figure \ref{fig:nongeom+} shows
  $(f_\psi^n)\inv(p_a)$ for $n = 6,7,8,9,10$.  Notice how the points
  in $(f^n)\inv(p_a)$ stay uniformly close together (since the stable
  tree for $\phi$ is geometric) where as in $(f_\psi^n)\inv(p_a)$ they
  start to clump together (since the stable tree for $\psi$ is
  nongeometric).

  \begin{figure}[p]
    \centering
    \includegraphics{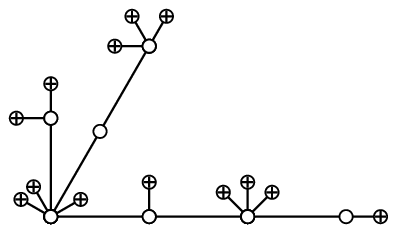}
    \includegraphics{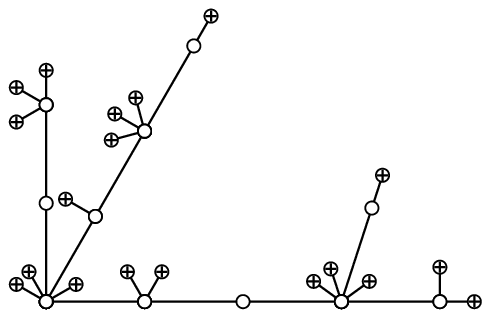}
    \includegraphics{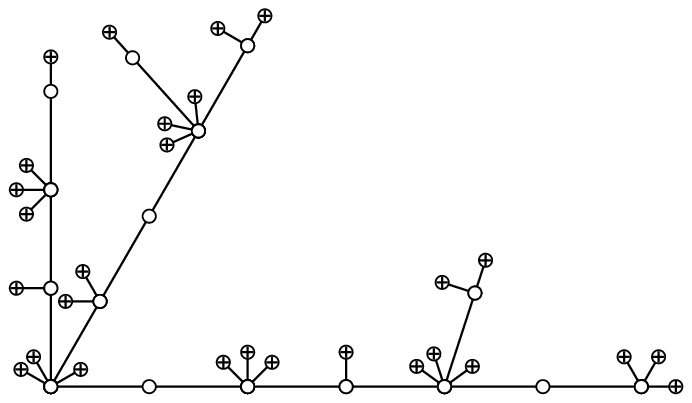}
    \includegraphics{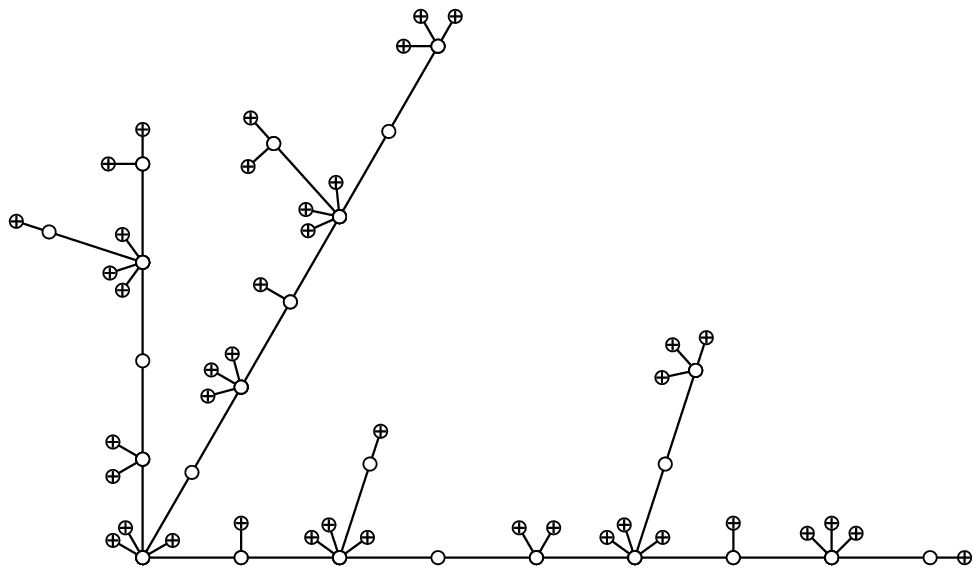}
    \includegraphics{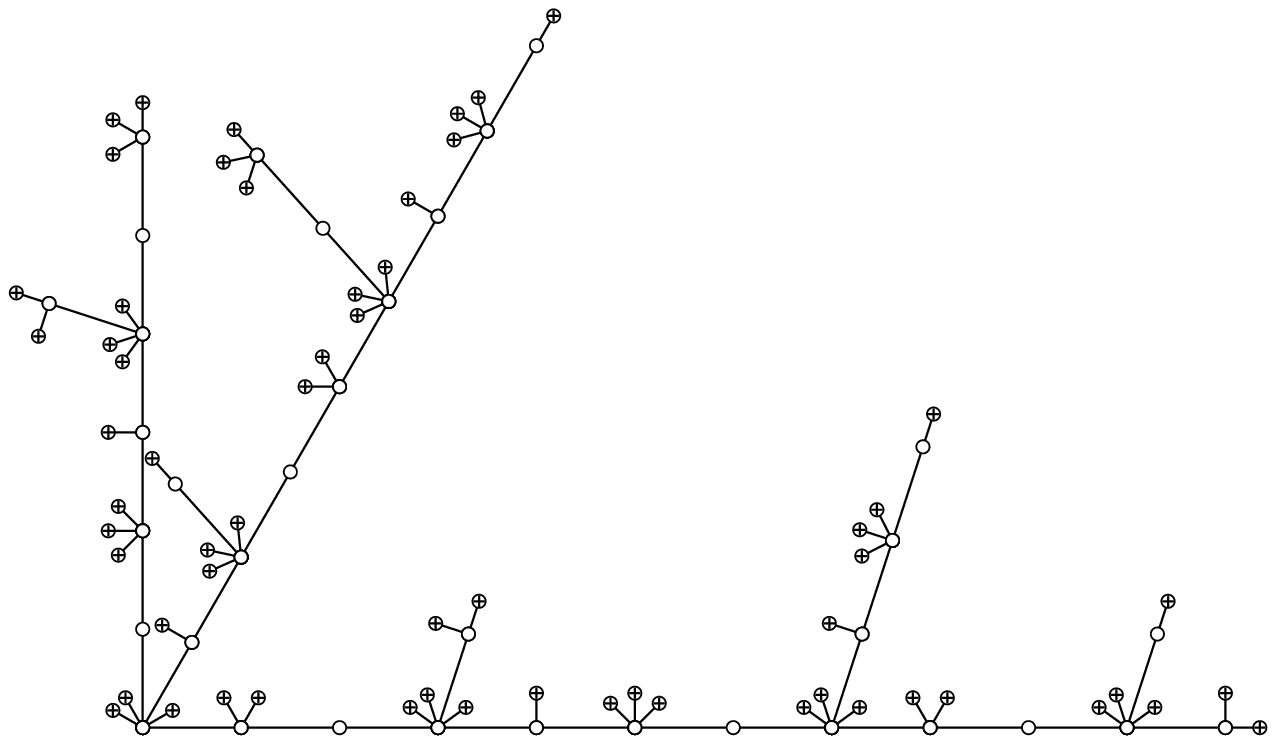}
    \caption{The trees $T^n_a$ in Example \ref{ex:geom-nongeom} for
      the automorphism with geometric stable
      tree.}\label{fig:parageom+}
  \end{figure}

  \begin{figure}[h]
    \centering
    \includegraphics{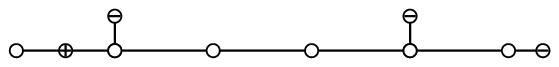}
    \includegraphics{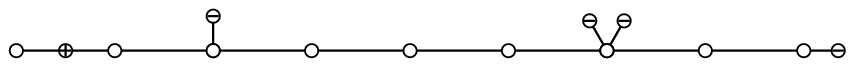}
    \includegraphics{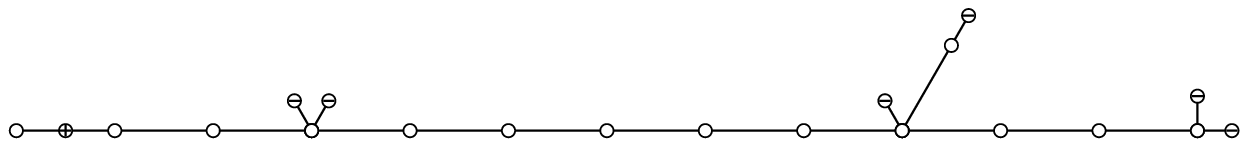}
    \includegraphics[width=5in]{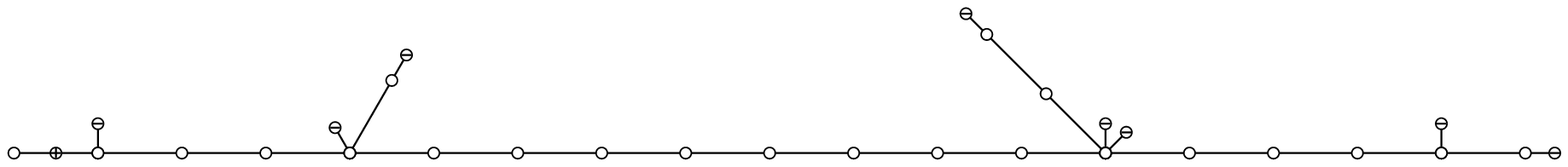}
    \includegraphics[width=5in]{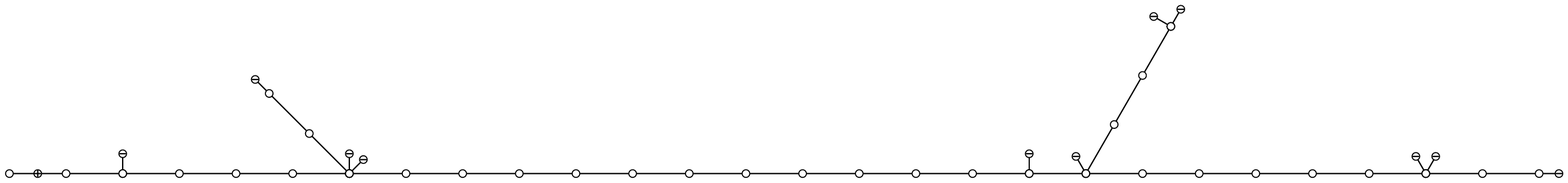}
    \caption{The trees $T^n_a$ in Example \ref{ex:geom-nongeom} for
      the automorphism with nongeometric stable
      tree.}\label{fig:nongeom+}
  \end{figure}
\end{example}


\bibliography{bibliography}

\begin{thebibliography}{10}

\bibitem{BKMM}
{\sc J.~Behrstock, B.~Kleiner, Y.~Minsky, and L.~Mosher}, {\em Geometry and
  rigidity of mapping class groups}.
\newblock \href{http://arxiv.org/abs/0801.2006}{arXiv:0801.2006}.

\bibitem{un:BF}
{\sc M.~Bestvina and M.~Feighn}, {\em Outer limts}.
\newblock preprint (1992).

\bibitem{ar:BFH97}
{\sc M.~Bestvina, M.~Feighn, and M.~Handel}, {\em Laminations, trees, and
  irreducible automorphisms of free groups}, Geom. Funct. Anal., 7 (1997),
  pp.~215--244.

\bibitem{ar:BH92}
{\sc M.~Bestvina and M.~Handel}, {\em Train tracks and automorphisms of free
  groups}, Ann. of Math. (2), 135 (1992), pp.~1--51.

\bibitem{Bonahon:ends}
{\sc F.~Bonahon}, {\em Bouts des vari\'et\'es hyperboliques de dimension
  {$3$}}, Ann. of Math. (2), 124 (1986), pp.~71--158.

\bibitem{ar:B06}
{\sc B.~H. Bowditch}, {\em Intersection numbers and the hyperbolicity of the
  curve complex}, J. Reine Angew. Math., 598 (2006), pp.~105--129.

\bibitem{col:BV06}
{\sc M.~R. Bridson and K.~Vogtmann}, {\em Automorphism groups of free groups,
  surface groups and free abelian groups}, in Problems on mapping class groups
  and related topics, vol.~74 of Proc. Sympos. Pure Math., Amer. Math. Soc.,
  Providence, RI, 2006, pp.~301--316.

\bibitem{ar:C87}
{\sc D.~Cooper}, {\em Automorphisms of free groups have finitely generated
  fixed point sets}, J. Algebra, 111 (1987), pp.~453--456.

\bibitem{un:CLS}
{\sc M.~Culler, G.~Levitt, and P.~Shalen}.
\newblock unpublished manuscript.

\bibitem{ar:CM87}
{\sc M.~Culler and J.~W. Morgan}, {\em Group actions on {${\bf R}$}-trees},
  Proc. London Math. Soc. (3), 55 (1987), pp.~571--604.

\bibitem{ar:CV86}
{\sc M.~Culler and K.~Vogtmann}, {\em Moduli of graphs and automorphisms of
  free groups}, Invent. Math., 84 (1986), pp.~91--119.

\bibitem{ar:FP06}
{\sc K.~Fujiwara and P.~Papasoglu}, {\em J{SJ}-decompositions of finitely
  presented groups and complexes of groups}, Geom. Funct. Anal., 16 (2006),
  pp.~70--125.

\bibitem{ar:G05}
{\sc V.~Guirardel}, {\em C\oe ur et nombre d'intersection pour les actions de
  groupes sur les arbres}, Ann. Sci. \'Ecole Norm. Sup. (4), 38 (2005),
  pp.~847--888.

\bibitem{ar:HM07}
{\sc M.~Handel and L.~Mosher}, {\em The expansion factors of an outer
  automorphism and its inverse}, Trans. Amer. Math. Soc., 359 (2007),
  pp.~3185--3208 (electronic).

\bibitem{un:HM}
\leavevmode\vrule height 2pt depth -1.6pt width 23pt, {\em Parageometric outer
  automorphisms of free groups}, Trans. Amer. Math. Soc., 359 (2007),
  pp.~3153--3183 (electronic).

\bibitem{Harer}
{\sc J.~L. Harer}, {\em Stability of the homology of the mapping class groups
  of orientable surfaces}, Ann. of Math. (2), 121 (1985), pp.~215--249.

\bibitem{Harer:MCGhomologystability}
\leavevmode\vrule height 2pt depth -1.6pt width 23pt, {\em Stability of the
  homology of the mapping class groups of orientable surfaces}, Ann. of Math.
  (2), 121 (1985), pp.~215--249.

\bibitem{Harer2}
\leavevmode\vrule height 2pt depth -1.6pt width 23pt, {\em The virtual
  cohomological dimension of the mapping class group of an orientable surface},
  Invent. Math., 84 (1986), pp.~157--176.

\bibitem{harvey:boundary}
{\sc W.~J. Harvey}, {\em Boundary structure of the modular group}, in Riemann
  Surfaces and Related Topics: Proceedings of the 1978 Stony Brook Conference,
  I.~Kra and B.~Maskit, eds., Ann. of Math. Stud. 97, Princeton, 1981.

\bibitem{ar:HV98}
{\sc A.~Hatcher and K.~Vogtmann}, {\em The complex of free factors of a free
  group}, Quart. J. Math. Oxford Ser. (2), 49 (1998), pp.~459--468.

\bibitem{ivanov:complexes2}
{\sc N.~V. Ivanov}, {\em Automorphisms of complexes of curves and of
  {T}eichm\"uller spaces}, Internat. Math. Res. Notices,  (1997), pp.~651--666.

\bibitem{col:K06}
{\sc I.~Kapovich}, {\em Currents on free groups}, in Topological and asymptotic
  aspects of group theory, vol.~394 of Contemp. Math., Amer. Math. Soc.,
  Providence, RI, 2006, pp.~149--176.

\bibitem{KapovichLustig}
{\sc I.~Kapovich and M.~Lustig}, {\em Geometric intersection number and
  analogues of the curve complex for free groups}, Geom. Topol., 13 (2009),
  pp.~1805--1833 (electronic).

\bibitem{korkmaz:complex}
{\sc M.~Korkmaz}, {\em Automorphisms of complexes of curves on punctured
  spheres and on punctured tori}, Topology Appl., 95 (1999), pp.~85--111.

\bibitem{ar:LP97}
{\sc G.~Levitt and F.~Paulin}, {\em Geometric group actions on trees}, Amer. J.
  Math., 119 (1997), pp.~83--102.

\bibitem{luo:complex}
{\sc F.~Luo}, {\em Automorphisms of the complex of curves}, Topology, 39
  (2000), pp.~283--298.

\bibitem{masurminsky:complex1}
{\sc H.~A. Masur and Y.~Minsky}, {\em Geometry of the complex of curves {I}:
  Hyperbolicity}, Invent. Math., 138 (1999), pp.~103--149.

\bibitem{ar:S98}
{\sc P.~Scott}, {\em The symmetry of intersection numbers in group theory},
  Geom. Topol., 2 (1998), pp.~11--29 (electronic).

\bibitem{ar:SS00}
{\sc P.~Scott and G.~A. Swarup}, {\em Splittings of groups and intersection
  numbers}, Geom. Topol., 4 (2000), pp.~179--218 (electronic).

\bibitem{ar:S83}
{\sc J.~R. Stallings}, {\em Topology of finite graphs}, Invent. Math., 71
  (1983), pp.~551--565.

\bibitem{ar:T88}
{\sc W.~P. Thurston}, {\em On the geometry and dynamics of diffeomorphisms of
  surfaces}, Bull. Amer. Math. Soc. (N.S.), 19 (1988), pp.~417--431.

\end{thebibliography}

\bibliographystyle{siam}

\end{document}